\documentclass[12pt,reqno]{amsart}
\usepackage{amsfonts,amsmath,amssymb,amsthm}
\usepackage{latexsym}
\usepackage{eucal}
\usepackage[ansinew]{inputenc}
\usepackage[american]{babel}
\usepackage{amsfonts}
\usepackage{amssymb}
\usepackage{graphicx}
\newcommand{\be}{\begin{equation}}
\newcommand{\ee}{\end{equation}}

\newtheorem{teo}{Theorem}[section]
\newtheorem{lema}{Lemma}[section]
\newtheorem{prop}{Proposition}[section]
\newtheorem{defi}{Definition}[section]
\newtheorem{obs}{Remark}[section]
\newtheorem{coro}{Corollary}[section]

\begin{document}

\title[stability of periodic waves.]
{stability properties of periodic traveling waves for the
Intermediate Long Wave Equation}


\thanks{{\it Date}: March, 2015.}

\maketitle

\begin{center}

{\bf Jaime Angulo Pava}

{\scriptsize  Institute of Mathematics and Statistics, State
University of
S\~ao Paulo, S\~ao Paulo, SP, Brazil.}\\
{\scriptsize angulo@ime.usp.br}

\vspace{3mm}

{\bf Eleomar Cardoso Jr.}

{\scriptsize   Federal University of Santa Catarina,  Blumenau, SC, Brazil}\\
{\scriptsize eleomar.jr@hotmail.com
}

\vspace{3mm}

{\bf F\'abio Natali}

{\scriptsize    Department of Mathematics, State University of
Maring\'a, Maring\'a, PR, Brazil. }\\
{\scriptsize fmanatali@uem.br}

\end{center}

\begin{abstract}
In this paper we determine orbital and linear stability of a class of spatially periodic
wavetrain solutions with the mean zero property related to the Intermediate Long
Wave equation. Our arguments follow the recent developments in  \cite{andrade-pastor}, \cite{natali} and \cite{DK} for the study of the
stability of periodic traveling waves.
\end{abstract}

\section{Introduction}

One of the most fascinating  phenomena given by nonlinear
dispersive equations is the existence of solutions that maintain their shape and traveling with constant speed. Such solutions
are caused by a perfect balance between the nonlinear and dispersive
effects at the medium. In general, these solutions are called
traveling waves and it is well known that the their existence has a
very wide applications in fluid dynamics, nonlinear optics,
hydrodynamic and many other fields (see pioneers works due to Boussinesq, Benjamin, Ono,
Benjamin-Bona-Mahoney, Miura, Gardner, and Kruskal).  Then, the study
concerning the dynamics related to these solutions has became one of
the important issues of the last decades for evolutive nonlinear
partial differential equations.

We can say that the initial impetus for the scientific activity of
these profiles was the inverse scattering theory (IST) for the
Korteweg-de Vries equation (KdV-equation henceforth)
$$
u_t+u_x+ (u^2)_x+u_{xxx}=0.
$$
One of the lessons learned by the IST is that the traveling wave
with a solitary wave profile, namely, $u(x,t)=\psi(x-ct)$ with $c>0$
and
$$
\lim_{|\xi|\to +\infty} \psi(\xi)=0,
$$
plays a central role in the long-time asymptotics of solutions to the
initial-value problem associated to KdV-equation. Indeed, general
classes of initial disturbances are known to solve into a finite
sequence of solitary waves followed by a dispersive tail. A
companion result is that individual solitary waves are orbitally
stable solutions of the evolution equation. The exact theory of
stability of solitary waves for the KdV-equation started by Benjamin
in \cite{Be} (see also Bona \cite{bona1}) whose maturity was reached
 a decade ago with the works due to  Albert \cite{albert1}, Albert and Bona \cite{AB}, Albert, Bona and Henry \cite{ABH}
 and Weinstein
\cite{W}-\cite{W1}. Next, in papers  due to Strauss \textit{at al.}
and Weinstein \cite{BSS}, \cite{grillakis1}, \cite{Wo} were shown that not all solitary-wave solutions are stable. Both
necessary and sufficient conditions for stability of the traveling
waves solutions of a range of nonlinear dispersive evolution
equations appear in various of the above references.

In the last years,  the study of  stability of traveling waves of
periodic type associated with nonlinear dispersive equations has
increased significantly.  A rich variety of new mathematical
problems have emerged, as well as, the physical importance related
to them. This subject is often studied in relation to the natural
symmetries associated to the model (translation invariance and/or
rotations invariance) and to perturbations of symmetric classes,
e.g., the class of periodic functions with  the same minimal period
as the underlying wave. In the case of shallow-water wave models (or
long internal waves in a density-stratified ocean, ion-acoustic
waves in a plasma or acoustic waves on a crystal lattice), it is
well known that a formal stability theory of periodic traveling wave
has started with the pioneering work of Benjamin  \cite{benjamin1}
regarding to the periodic steady solutions called {\it cnoidal
waves} for the KdV equation. The waveform profiles were found first by Korteweg and de-Vries
 for KdV-equation. The cnoidal traveling wave solution, namely,
 $u(x,t)=\varphi_c(x-ct)$ has a profile given by
\begin{equation}\label{cnoidal1}
\varphi_c(\xi)=\beta_{2}+
(\beta_3-\beta_2)\mbox{cn}^2\left({\sqrt{\frac{\beta_3-\beta_1}{12}}\xi};k\right),
\end{equation}
where $cn(\cdot;k)$ represents the Jacobi
elliptic function called {\it cnoidal} associated with the elliptic modulus $k\in (0,1)$ and
$\beta_i$'s are real constants satisfying the classical relations
\begin{equation}\label{betas}
\beta_1<\beta_2<\beta_3, \;\; \beta_1+\beta_2+\beta_3=3c,\;\; k^2=\frac{\beta_3-\beta_2}{\beta_3-\beta_1}.
\end{equation}
We recall that $\varphi_c$ satisfies the second order differential equation
\begin{equation}\label{kdveq}
-\varphi''_c(\xi)+c \varphi_c(\xi)-\frac12 \varphi^2_c(\xi)=A_{c},\quad \xi\in\mathbb R
\end{equation}
with $A_{\varphi_c}=-\frac16\sum_{i<j}\beta_i\beta_j$, and that the
formula (\ref{cnoidal1}) is deduced from the  theory of elliptic
integrals and  elliptic  functions. The existence of smooth solutions for (\ref{kdveq}) with a minimal period $L$, $c\in
I\subset \mathbb R\to \varphi_c\in H^n_{per}([0,L])$ is determined
from the implicit function theorem. The  interval $I$ in general
depends of qualitative properties of $\varphi_c$, for instance, for
the property of mean zero, $\int_0^L \varphi_c(\xi)d\xi=0$, we have
$I=(0,+\infty)$ and for $A_{\varphi_c}=0$ and $\varphi_c(\xi)>0$ for
all $\xi\in \mathbb R$, we have $I=(\frac{4\pi^2}{L^2}, +\infty)$. A
first stability approach for the cnoidal wave profile
(\ref{cnoidal1}) was began by Benjamin in \cite{benjamin1} regarding
the stability in $H^1_{per}([0,L])$ of the orbit
\begin{equation}\label{orbit}
\Omega_{\varphi_c}=\{\varphi_c(\cdot+y): y\in \mathbb R\},
\end{equation}
by the periodic flow of the KdV equation. But only years later a
complete study was carried out by Angulo, Bona and Scialom in \cite{ABS}
(see also \cite{A}).

 Recently, Angulo and Natali in \cite{natali} (see also \cite{A}) have  established a new approach for
studying the stability of even and positive periodic traveling waves solutions associated to the  general dispersive model
\begin{equation}
u_t+2uu_x-(\mathcal{M}u)_x=0, \label{equakawa}
\end{equation}
where $\mathcal{M}$ is a differential or
pseudo-differential operator in the framework of periodic functions.
$\mathcal{M}$ is defined as a Fourier multiplier operator by
\begin{equation}\label{symbol}
\widehat{\mathcal{M}g}(n)=\theta
(n)\widehat{g}(n),\;\;\;\kappa\in \mathbb Z,
\end{equation}
where the symbol $\theta$ of $\mathcal{M}$ is assumed to be a
mensurable, locally bounded function on $\mathbb R$, satisfying the
condition
\begin{equation}\label{alpha1}
a_1|n|^{m_1}\leq \theta(n)\leq a_2 (1+|n|)^{m_2},
\end{equation}
where $m_1\leq m_2$, $|n|\geqq n_0$, $\theta(n)>b$ for all $n\in \mathbb Z$, and $a_i\geq 0$. One of the advantage of Angulo and Natali approach was the  possibility of studying non-local evolution models in a periodic framework. For instance, let us consider the case of the Benjamin-Ono equation (henceforth BO-equation)
\begin{equation}\label{BO}
u_t+uu_x-\mathcal{H}u_{xx}=0,
\end{equation}
with  $\mathcal{H}$ denoting the periodic Hilbert transform and  defined for $L$-periodic functions $f$ as
\begin{equation}\label{symbBO}
\mathcal{H}f(x)=\frac{1}{L} \text{p.v.} \int_{-L/2}^{L/2}
cot\Big[\frac{\pi(x-y)}{L}\Big]f(y)dy,
\end{equation}
where p.v. represents the Cauchy principal value of the integral, we
have that the Fourier transform of $\mathcal{H}f$ is given by the
sequence $\{\widehat{\mathcal{H}f}(n)\}_{n\in \mathbb Z}$, where
$\widehat{\mathcal{H}f}(n)=-i sgn(n) \widehat{f}(n)$. In other
words, we have that $\mathcal{M}=\mathcal{H}\partial_x$ whose symbol
is $\theta(n)=|n|$. The periodic traveling waves
$u(x,t)=\varphi_c(x-ct)$ for the BO-equation with minimal period $L$
satisfies the following non-local pseudo-differential equation
$$
\mathcal{H}\phi_c+c\phi_c-\frac12 \phi_c^2=0,
$$
and they are given by
$$
\phi_c(x)= \frac{4\pi}{L} \frac{senh(\gamma)}{cosh(\gamma)-cos(\frac{2\pi x}{L})}
$$
where $\gamma >0$ satisfies $tanh(\gamma)=\frac{2\pi}{cL}$
(therefore the wave speed  $c$ must satisfy $c>2\pi/L$). As an
application of the theory  in \cite{natali}, the authors obtained the
first nonlinear stability result for the orbit generated by the wave
$\varphi_c$.

In this paper,  we are interested in studying  the orbital and linear stability of a family
periodic traveling waves for the physically relevant Intermediate
Long Wave equation (ILW equation  henceforth),
\begin{equation}\label{ILW}
u_t+2uu_x+\delta^{-1}u_x-(\mathcal{T_\delta}u)_{xx}=0,\ \ \ \ \ \ \delta>0,
\end{equation}
with $u=u(x,t)$  a $L-$periodic function and $x,t\in\mathbb R$. The linear operator
$\mathcal{T_\delta}$ is defined by
$$
\mathcal{T_\delta}u(x)=\frac{1}{L} \text{p.v.} \int_{-L/2}^{L/2}
\Gamma_{\delta, L}(x-y) u(y)dy,
$$
where
$$
\Gamma_{\delta, L}(\xi)=-i\sum_{n\neq 0} \coth\left(\frac{2\pi
n\delta}{L}\right)e^{2in \pi \xi/L}.
$$
Actually, the physical derivation of (\ref{ILW}) in a periodic setting requires that
$$
\int_{-L/2}^{L/2} u(x)dx=0,
$$
 where we always can impose (\ref{ILW}), because any non-zero mean could be removed by the
Galilean transformation $v(x,t)=u(x+2\gamma t, t)-\gamma$, $\gamma\in \mathbb R$. Hence,  from the
theory of elliptic functions (see Ablowitz, \textit{et al.}
\cite{afss}) we obtain that
$$
\mathcal{T_\delta}u(x)=-i\sum_{n\neq 0} \coth\left(\frac{2\pi n\delta}{L}\right)\widehat{u}(n)e^{2in \pi \xi/L}.
$$
Moreover, for $\delta\to \infty$, $L$ fixed, we have (see \cite{afss})
$$
\lim_{\delta\to \infty} \Gamma_{\delta, L}(\xi)=-cot \Big( \frac{\pi
x}{L}\Big),
$$
which is the kernel of the Hilbert transform in $(\ref{symbBO})$.
Therefore, the ILW equation (\ref{ILW}) is the natural periodic
extension of the BO-equation (\ref{BO}). We note that  the ILW equation is an example of the class of dispersive
models (\ref{equakawa}) with exactly $\mathcal
M_\delta=\mathcal{T_\delta}\partial_x-\frac{1}{\delta}$.

Now, one of our main objectives in this paper, it will be   to find periodic solutions for (\ref{ILW}) of the form
$u(x,t)=\varphi_c(x-ct)$ with the periodic profile $\varphi_c$
having an mean zero and satisfying \be\label{travkdv}
-c\varphi_c+\varphi_c^2-\mathcal{M}_{\delta}\varphi_c=A_c, \ee
 where $A_c$  will be an integration constant given by $A_c=\frac{1}{L}\int_0^L\varphi_c^2(x)dx$.
 In section 3 we obtain, as a consequence of Theorem \ref{exis},  the following property associated  to the pseudo-differential equation (\ref{travkdv}):

\begin{itemize}
\item[$(P0)$] There is a smooth curve of even periodic solutions for (\ref{travkdv})
with the mean zero property, in the form
$$
c\in I\subset \mathbb R\mapsto\varphi_{c}\in H_{per}^n([0,L]),\ \
n\in \mathbb N,
$$
all of them with the same minimal period $L>0$.
\end{itemize}

By following the  arguments due to Parker \cite{parker} (see also
Nakamura and Matsuno in \cite{NM}), we obtain the following  formula of
even periodic solution for  (\ref{travkdv}) with the mean zero
property (see section 3 below),
\begin{equation}\label{solilw}
\varphi_c(x):=\varphi_c(L,\delta,k;x)
=\displaystyle\frac{2K(k)i}{L}\displaystyle\left[Z\displaystyle
\left(\displaystyle\frac{2K(k)}{L}(x-i\delta);k\right)-Z\displaystyle
\left(\displaystyle\frac{2K(k)}{L}(x+i\delta);k\right)\right],\ \
\end{equation}
where $K(k)$ denotes the complete elliptic integral of the first
kind, $Z$ is the Jacobi Zeta Function and $k\in(0,1)$ (see notation section below). For fixed $L$
and $\delta$, the wave-speed $c$ and the elliptic modulus $k$ must
satisfy specific restrictions.

Other one focus of our  study, it will be  the dynamic
of solutions of the ILW equation initially close to the
mean-zero profile $\varphi_c$ in (\ref{solilw}), the stability of the profile  $\varphi_c$. There are two
common approaches to the stability question. Firstly, we can
analyze the nonlinear initial-value problem governing the difference
between an arbitrary solution of the ILW equation and a given exact
solution representing a wavetrain, the profile $\varphi_c$. In the
first approximation, we assume that the difference is
small and we linearize the evolution equation. The resulting linear equation can be
studied in an appropriate frame of reference by a spectral approach.
To our knowledge, the linearized spectral approach has  never been established
 for the ILW equation. A second approach to stability is
the  orbital stability, more exactly, we study the Lyapunov stability
property of the orbit
\begin{equation}\label{orbit1}
\Omega_{\varphi_c}=\{\varphi_c(\cdot+y): y\in \mathbb R\},
\end{equation}
generated by the profile  $\varphi_c$. The study of the dynamic of the set $\Omega_{\varphi_c}$ consist  in
verifying that for any initial condition $u_0$  close to
$\Omega_{\varphi_c}$ we have that the solution $u(t)$ of
$(\ref{ILW})$ with $u(0)=u_0$ remains close to $\Omega_{\varphi_c}$
for all values of $t\in\mathbb{R}$. The specific notion of ``close''
is based in terms of the following pseudo-metric defined on a
determined space $W$, namely,  for $f,g\in W$,
\begin{equation}\label{dis}
d_2(f,g)=\inf_{r\in \mathbb R}\|f- \tau_r g\|_W,
\end{equation}
with $\tau_rh(x)=h(x+r)$. The translation symmetry $\tau$ enables us to form a quotient
space, $W/\tau$, by identifying the translations $\tau f$ of each
$f\in W$. If we consider $f$ and $g$ as elements of $W/\tau$, we obtain
that $d_2$ represents a well-defined  {\it metric} on this set. Note that in
$W/\tau$ the difference $u-\varphi_c$, between $\varphi_c$ and the
perturbed solution $u$, it will represent the most vital difference between
two wave forms, namely, the {\it shape}. Again, according to our best knowledge, the
orbital stability property associated to the profile $\varphi_c$ in (\ref{solilw}) has  never been established for the ILW equation in a periodic setting.

Next, we shall give a brief explanation of our work. In fact, let us consider the new variable
$$
v(x,t)=u(x+xt,t)-\varphi_c(x),
$$
where $u$ solves $(\ref{ILW})$ and $\varphi_c$ solves $(\ref{travkdv})$. Substituting this form in equation $(\ref{ILW})$ and by using
(\ref{travkdv}) one finds that $v$ satisfies the nonlinear equation
\be\label{vequ}
v_t+2vv_x+2(v\varphi_c)_x-cv_x-\mathcal{M}_{\delta}v_x=0. \ee As a
leading approximation for small perturbation, we replace
(\ref{vequ}) by its linearization about $\varphi_c$, and hence
obtain the linear equation \be\label{vlinear}
v_t=\partial_x(\mathcal{M}_{\delta} v +cv -2v\varphi_c). \ee Since
$\varphi_c$ depends only on $x$, the equation
(\ref{vlinear}) admits treatment by separation of variables, which
leads naturally to a  spectral problem. Then,  by seeking particular solutions
of (\ref{vlinear}) of the form $v(x,t)=e^{\lambda t} \psi(x)$, where
$\lambda \in \mathbb C$, $\psi$ satisfies  the linear problem
\be\label{vlinear2}
\partial_x\mathcal{L}\psi=\lambda \psi,
\ee
for $\mathcal{L}:=\mathcal{L}_{c,\delta}$ denoting  the self-adjoint operator
 \be\label{linop}
\mathcal{L}_{c,\delta}:=\mathcal{M}_{\delta}+c-2\varphi_c. \ee We
recall that the complex growth rate $\lambda$ appears as (spectral)
parameter. Equation (\ref{linop}) will only have a nonzero solution
$\psi$ in a given Banach space $Y$ for certain $\lambda \in \mathbb
C$. A necessary condition for the stability of $\varphi_c$ is that
there are not  points $\lambda$ with $\mbox{Re}(\lambda)>0$ (which
would imply the existence  of a solution $v$ of (\ref{vlinear}) that
lies in $Y$ as a function of $x$ and grows exponentially in time).
If we denoted  by $\sigma$ the spectrum of $\partial_x\mathcal{L}$,
the later discussion suggests the utility of the following
definition:

\begin{defi} (spectral stability and instability)
\label{defspe} A periodic traveling wave solution $\varphi_c$ of the
ILW equation (\ref{ILW}) is said to be spectrally stable if
$\sigma\subset i\mathbb R$. Otherwise (i.e., if $\sigma$ contains
point with $\mbox{Re}(\lambda)>0$) $\varphi_c$ is  {\it spectrally
unstable}.
\end{defi}

We recall that as (\ref{vlinear}) is a real Hamiltonian equation, it
forces certain elementary symmetries on the spectrum of $\sigma$,
more exactly, $\sigma$ is symmetric with respect to reflection in
the real and imaginary axes. Therefore, it implies that
exponentially growing perturbation are always paired with
exponentially decaying ones. It is the reason by which was only
required in  Definition \ref{defspe} that the spectral parameter
$\lambda$ satisfies that $\mbox{Re}(\lambda)>0$.

An similar spectral problem to (\ref{vlinear2}) has been the focus
of many research studies recently. For instance, if we restrict initially to
traveling wave solution of solitary wave type, sufficient conditions
in order to get the linear stability/instability  has been
established for many specific dispersive equations in
Kapitula and Stefanov \cite{kap}, in particular, the linear
stability related to the generalized Korteweg-de Vries equation
\begin{equation}\label{gkdv}
u_t+(p+1)u^pu_x+u_{xxx}=0\qquad p\in\mathbb N,
\end{equation}
 was obtained by using a Krein-Hamiltonian instability index to count the number of negative
eigenvalues with a positive real part.  In the case of linear instability, Lin in  \cite{lin} and Lopes in \cite{lopes} have
presented sufficient conditions  for general dispersive models.

\indent In a periodic framework, general spectral problem of the form
$$
J\mathcal L\psi=\lambda \psi
$$
 has emerged, with $J=\partial_x$ and $\mathcal L$ a self-adjoint operator. Since $J$ is not a one-to-one operator, classical linear stability results as in \cite{grillakis1}  can not be applied. To overcome this difficult, recently Deconinck and Kapitula in \cite{DK} (see also Haragus and Kapitula \cite{haragus})
considered  the similar problem
\be\label{modspecp1}
J\mathcal{L}\big|_{H_0}\psi=\lambda \psi,
 \ee
in the closed subspace of mean zero,
\begin{equation}\label{zero}
H_0=\left\{f\in L^2([0,L]);\ \int_0^Lf(x)dx=0\right\}.
\end{equation}
Thus,  an specific Krein-Hamiltonian index formula was deduced for concluding the linear stability of periodic profile with a mean zero property. In particular, it was  deduced the linear
stability of periodic  traveling waves of  cnoidal type associated with the equation
$(\ref{gkdv})$ for $p=2$ (we also refer the reader to see Bronski, Johnson and Kapitula in \cite{BJK}  and Deconinck and Nivala in \cite{DK}). We note, nevertheless, that for obtaining this specific result was necessary to know the periodic wave profile as well as the
knowledge  of a specific quantity of  eigenvalues associated to the Lam\'e problem
$$
-\Phi''+6k^2sn^2(x;k)\Phi=\theta \Phi.
$$
 Unfortunately, in our problem (\ref{vlinear2}), this specific type of information can not be established.

We note that the spectral/orbital stability properties of periodic traveling waves in Hamiltonian equations that are first-order in time (e.g. the Korteweg-de Vries or the Schr\"odinger equations) have been very  well-studied in recent years by using different approaches to those discussed above. See, for instance, Bronski and Johnson \cite{BJ}, Bronski, Johnson and Kapitula \cite{BJK}-\cite{BJK2}, Bronski, Johnson and Zumbrun \cite{BJZ}, Deconinck and Kapitula \cite{DK1}, Deconinck and Nivala \cite{DN}, Haragus and Kapitula \cite{haragus}, Hur and Johnson \cite{HJ}, Jonhson \cite{J2}-\cite{J1} and Kapitula and Promislow \cite{KP}.

In section 5 below, we use the approaches in Angulo and Natali
\cite{AN1}, Deconinck and Kapitula \cite{DK} and Haragus and Kapitula
\cite{haragus} for establishing the relevant result that the
periodic profile $\varphi_c$ in (\ref{solilw}) for the ILW equation
are linearly  stable. By techniques reasons, we establish it result for $c$ being strictly positive (see Remarks  \ref{restri} and \ref{numeric} below).

Now, some informations for obtaining our  linear stability result  in section 5 for $\varphi_c$  in (\ref{solilw}) can be  used in order to conclude the orbital stability property of these periodic waves. Moreover, it property will be established for every admissible  speed-wave $c$. Our approach,  it will follow from
a slight adaptation of the classical Lyapunov stability analysis established  by Andrade and Pastor in \cite{andrade-pastor}. In our case, the stability analysis will be based on the elliptic modulus $k$ instead of $c$, such as is standard in the classical literature, therefore we need establish a stability framework adapted to this new ``speed-wave''. The
energy space where the orbital stability property of the profile $\varphi_c$ will be studied, it  is the following Hilbert-space,
\begin{equation}\label{W}
\mathcal{W}=\left\{g\in L_{per}^2([0,L]);\
||g||_{\mathcal{W}}:=\left(\sum_{m=-\infty}^{+\infty}[1+\theta_{\delta}(m)]|\widehat{g}(m)|^2\right)^{\frac{1}{2}}<\infty\right\},
\end{equation}
where $\theta_{\delta}$ indicates the symbol associated with
$\mathcal{M}_{\delta}$.  In section 6, we briefly describe the main arguments for obtaining our orbital result of the profile $\varphi_c$
by  the periodic flow of the ILW-equation.\\
\indent Our paper is organized as follows. In section 2 we present notation and  the definition of the Jacobi elliptic functions. Section 3 is devoted
to the existence of periodic waves having the mean zero property. In
section 4, we present the required spectral property associated with
the linear operator $(\ref{linop})$ by following the arguments in
\cite{natali}. In section 5, the linear stability of the periodic profile $\varphi_c$  will
be shown. To the end, in section 6 we establish our  orbital stability  result.

 \section{Notation}

 For $k\in(0,1)$, we define the \textit{normal elliptic integral of the first kind},
 $$
u(x;k)=\int\limits_{0}^{x}\frac{dt}{\sqrt{(1-t^{2})(1-k^{2}t^{2})}}=\int\limits_{0}^{\varphi}\frac{d\theta}{\sqrt{1-k^{2}sin^{2}\theta}}=F(\varphi;k)
$$
with $x=sin \varphi$. The number $k$ and $\varphi$ are called  the  \textit{modulus} and the \textit{argument}, respectively.  For $x=1$ ($\varphi=\frac{\pi}{2}$), the integral above is said to be \textit{complete}. In this case, ones writes :
$$
K(k)=\int\limits_{0}^{1}\frac{dt}{\sqrt{(1-t^{2})(1-k^{2}t^{2})}}=\int\limits_{0}^{\frac{\pi}{2}}\frac{d\theta}{\sqrt{1-k^{2}sin^{2}\theta}}.
$$
Hence, $K(0)=\frac{\pi}{2}$ and $K(1)=+\infty$. For $k$ fixed,  $u=u(x;k)$ is a strictly increasing function of variable $x$ (real). We define its inverse function by $x\equiv sn(u;k)$ (\textit{snoidal} function).  Then, we obtain the basic Jacobian elliptic functions \textit{cnoidal} and \textit{dnoidal}, defined by
$cn(u;k)\equiv\sqrt{1-sn^{2}(u;k)}$ and $dn(u;k)\equiv\sqrt{1-k^{2}sn^{2}(u;k)}$
 (see Byrd and Friedman \cite{byrd} and Abramowitz and Segun \cite{abramo}).  Snoidal, cnoidal,  and  dnoidal have fundamental period $4K(k)$, $4K(k)$ and $2K(k)$, respectively. Moreover,  $sn^{2}(u;k)+cn^{2}(u;k)=1$,   $k^2sn^{2}(u;k)+dn^{2}(u;k)=1$, $sn(u;0)=sin(u)$, $cn(u; 0)=cos(u)$, $sn(u ;1)=tanh(u)$ and $cn(u ;1)=dn(u,1)=sech(u)$.  The Zeta Jacobi function, $Z(u)=Z(u,k)$, it is  defined for $u\in \mathbb R$ by
$$
Z(u)=\int\limits_{0}^{u}\bigg[dn^{2}(x;k)-\frac{E(k)}{K(k)} \bigg]dx.
$$
It is a function which  is odd with fundamental period $2K(k)$. Moreover, $Z(\pi/2, k)=0$ and $Z(mK)=0$, para $m=0,1,2,...$.  For $u$ being a complex argument we refer the reader to formula 143.01 in \cite{byrd}. In particular for $u=ix$, $x\in \mathbb R$ we obtain
$$
Z(ix,k)=i\frac{sn(x; k')}{cn(x; k')} dn(x; k')-iZ(x, k ')-i\frac{\pi x}{2K(k)K(k')},
$$
with $k'=\sqrt{1-k^2}$.

\section{Existence of Periodic Waves.}

\setcounter{equation}{0}
    \setcounter{defi}{0}
    \setcounter{teo}{0}
    \setcounter{lema}{0}
    \setcounter{prop}{0}
    \setcounter{coro}{0}

 This section is devoted to establish the property $(P0)$ defined in the introduction, more exactly, we construct a smooth curve of periodic waves with the mean zero property,
$c\in I\mapsto\varphi_{c}\in H_{per}^s([0,L])\cap H_0$,  where the period $L>0$ and the velocity $c$ will have some specific restrictions. Our arguments will follow Hirota's method, put forward in the works \cite{NM}  and \cite{parker}. By convenience of the reader and from our stability approach to be established in sections 5 and 6, we will review slightly the method.

Indeed, let us assume the existence of
$f:\mathbb{C}\times\mathbb{R}\rightarrow\mathbb{C}$, such that
the profile
$$
u(x,t)=i\displaystyle\frac{\partial}{\partial
x}\displaystyle\left[\ln\displaystyle\left(\displaystyle\frac{f(x+i\delta,t)}{f(x-i\delta,t)}\right)\right],\
(x,t)\in\mathbb{R}\times \mathbb{R},
$$
it will  satisfy equation $(\ref{ILW})$, with $f(\cdot, t)$ being analytic in a specific rectangle $R$ of the complex-plane. To simplify the notation, we  define $f_{+}(x,t)=f(x+i\delta,t)$ and $f_{-}(x,t)=f(x-i\delta,t)$. So, by  arguments in \cite{parker},  there is a constant
$B$, such that we have the bilinear equation
\begin{eqnarray}\label{ilw.25.1}\displaystyle\left[iD_t+\displaystyle\frac{i}{\delta}D_x-D_x^2+B\right]f_{+}\cdot
f_{-}=0,\end{eqnarray}
with
 $$D_t^mD_x^na(x,t)\cdot
b(x,t):=\displaystyle\left.(\partial_t-\partial_{t'})^m(\partial_x-\partial_{x'})^na(x,t)b(x',t')\right|_{x=x',\
t=t'}.$$ In addition, we can deduce from (\ref{ilw.25.1}) that
\begin{eqnarray}\label{ilw.25.2}F(D_t,D_x)f\cdot f=0,\end{eqnarray} where $$F(D_t,D_x)\equiv i\displaystyle\left(D_t+\displaystyle\frac{1}{\delta}D_x\right)\sinh(i\delta D_x)+(D_x^2-B)\cosh(i\delta D_x).$$

Consider $z=px+w t$, where $p,w\in\mathbb{R}$ will be determined later.
Suppose that $f$ has the following Jacobi Theta profile (see \cite{abramo})
$$
f(x,t)\equiv \theta_3(z,q):=1+2\displaystyle\left[\displaystyle\sum_{n=1}^{+\infty}q^{n^2}\cos(2nz)\right]=\displaystyle\sum_{n=-\infty}^{+\infty}q^{n^2}e^{2inz},
$$
for $q=e^{i\pi\tau}$ with $\tau=i\frac{K(k')}{K(k)}$, where $K'(k)\equiv K(\sqrt{1-k^2})$ is the associated elliptic integral of the first kind. In general $q=q(\tau)$ is the function called
``nome''  with $\textrm{Im}(\tau)>0$. By substituting
$f$ at the identity (\ref{ilw.25.2}), one has
$$\widetilde{F}_0\theta_3(2z,q^2)+\widetilde{F}_1q^{-\frac{1}{2}}\theta_2(2z,q^2)=0.$$
Here, $\theta_2$ represents the  Jacobi Theta function of second kind. Moreover, one has $$\widetilde{F}_m=\displaystyle\sum_{n=-\infty}^{+\infty}F[2i(2n-m)w,2i(2n-m)p]q^{n^2+(n-m)^2},\ \
m=0,1.$$

In order to prove that $f(x,t)=\theta_3(z,q)$ is a periodic solution related to the equation
(\ref{ILW}), it is enough to prove that
$\tilde{F}_0=\tilde{F}_1=0$. To do so, it suffices to show that
\begin{eqnarray}\label{ilw.25.2.0}\displaystyle\frac{1}{\delta}\displaystyle\left(w+\displaystyle\frac{p}{\delta}\right)A_0'-\displaystyle\frac{p^2}{\delta^2}A_0''-A_0B=0\
\ \textrm{e}\ \
\displaystyle\frac{1}{\delta}\displaystyle\left(w+\displaystyle\frac{p}{\delta}\right)A_1'-\displaystyle\frac{p^2}{\delta^2}A_1''-A_1B=0,\
 \ \end{eqnarray} where
\begin{eqnarray*}A_0=A_0(p;q,\delta)=\displaystyle\sum_{n=-\infty}^{+\infty}q^{2n^2}\cosh(4np\delta)=\theta_3(2ip\delta,q^2),\end{eqnarray*}
\begin{eqnarray*}A_1=A_1(p;q,\delta)=\displaystyle\sum_{n=-\infty}^{+\infty}q^{n^2+(n-1)^2}\cosh[2(2n-1)p\delta]=q^{\frac{1}{2}}\theta_2(2ip\delta,q^2)\end{eqnarray*}
and $A_i'$, $i=0,1$, represent the derivative of the parameters $A_0$ e $A_1$ with respect to $p$, respectively. Next, we fix parameters $p$, $q$ and $\delta$ above. Solving the system in (\ref{ilw.25.2.0}) we get
\begin{eqnarray*}B=B(p;q,\delta)=\displaystyle\frac{p^2}{\delta^2}\cdot\displaystyle\frac{A_0'A_1''-A_0''A_1'}{A_0A_1'-A_0'A_1}\end{eqnarray*}
and
\begin{eqnarray*}w=w(p;q,\delta)=-\displaystyle\frac{p}{\delta}+\displaystyle\frac{p^2}{\delta}\cdot\displaystyle\frac{A_0A_1''-A_0''A_1}{A_0A_1'-A_0'A_1}=-\displaystyle\frac{p}{\delta}+\displaystyle\frac{p^2}{\delta}\cdot\displaystyle\frac{\partial}{\partial
p}\displaystyle\left\{\ln[W(A_0,A_1)]\right\},\end{eqnarray*} where
$W(A_0,A_1)=A_0A_1'-A_0'A_1$ indicates the Wronskian of $A_0$ and $A_1$. Now, if we use some standard identities concerning the Jacobi elliptic functions (see \cite{abramo} and \cite{byrd}), we deduce that $f(x,t)=\theta_3(z,q)$ must satisfy the identity
(\ref{ilw.25.2}) provided that
\begin{eqnarray*}B=B(p;q,\delta)=-p^2\displaystyle\left[\displaystyle\frac{\theta_1''(2ip\delta,q)}{\theta_1(2ip\delta,q)}-\displaystyle\frac{\theta_1'''(0,q)}{\theta_1'(0,q)}\right]\end{eqnarray*}
and
\begin{eqnarray*}w=w(p;q,\delta)=-\displaystyle\frac{p}{\delta}+2ip^2\cdot\displaystyle\frac{\theta_1'(2ip\delta,q)}{\theta_1(2ip\delta,q)},
\end{eqnarray*}
where $\theta_1$ represents the  Jacobi Theta function of first kind.

Now, similar arguments can be used if one considers the slight change of variables $z\mapsto\displaystyle\frac{z}{2}$. In this case we see that
\begin{eqnarray}\label{ilw.25.2.2}B=B(p;k,\delta)=-\displaystyle\frac{p^2}{4}\cdot\displaystyle\left[\displaystyle\frac{\theta_1''(ip\delta,q(k))}{\theta_1(ip\delta,q(k))}-\displaystyle\frac{\theta_1'''(0,q(k))}{\theta_1'(0,q(k))}\right]\end{eqnarray}
and
\begin{eqnarray}\label{ilw.25.2.1}w=w(p;k,\delta)=-\displaystyle\frac{p}{\delta}+ip^2\cdot\displaystyle\frac{\theta_1'(ip\delta,q(k))}{\theta_1(ip\delta,q(k))},\end{eqnarray}
where $k\in (0,1)$, $k'=\displaystyle\sqrt{1-k^2}$, and
$q(k)=e^{-\frac{\pi K(k')}{K(k)}}$.

Hence, we obtain that our hypothetic  solution $u$ becomes
\begin{eqnarray}\label{ilw.25.3}u(x,t)&=&i\displaystyle\frac{\partial}{\partial
x}\displaystyle\left\{\ln\displaystyle\left[\displaystyle\frac{\theta_3\displaystyle\left(\frac{1}{2}(z-ip\delta),q(k)\right)}{\theta_3\displaystyle\left(\frac{1}{2}(z+ip\delta),q(k)\right)}\right]\right\}\nonumber\\
\\
&=&\displaystyle\frac{ip}{2}\cdot\displaystyle
\left[\displaystyle\frac{\theta_3'\displaystyle\left(\frac{1}{2}(z-ip\delta),q(k)\right)}
{\theta_3\displaystyle\left(\frac{1}{2}(z-ip\delta),q(k)\right)}-
\displaystyle\frac{\theta_3'\displaystyle\left(\frac{1}{2}(z+ip\delta),q(k)\right)}{\theta_3\displaystyle\left(\frac{1}{2}(z+ip\delta),q(k)\right)}
\right],
\nonumber\end{eqnarray}
it which represents a $L$-periodic function at the spatial variable with the natural choice of $p=2\pi/L$.

Next, we obtain specific restrictions on the parameter $p, k$ and the minimal period $L$ for $u$ to be a smooth periodic function. Indeed, for $k\in(0,1)$ fixed, it is well known that  the  theta function $\theta_3(z,q(k))$ has simple zeros at the points
$$z=\displaystyle\left(m+\displaystyle\frac{1}{2}\right)\pi+\displaystyle\left(n+\displaystyle\frac{1}{2}\right)\pi\tau, \quad m,n\in\mathbb{Z}.
$$
So, the right-hand side of $(\ref{ilw.25.3})$ possess infinitely many isolated singularities which we need to avoid. To overcome this situation, it makes necessary to impose a convenient condition over the parameters $p$, $\delta$ and $k$, namely,
\begin{eqnarray}\label{ilw.25.4}0<p\delta<-i\pi\tau=\pi\displaystyle\frac{K(k')}{K(k)},
\end{eqnarray}
$k'=\sqrt{1-k^2}$. To do so, it suffices to consider $k\in(0,1)$ satisfying
\begin{equation}\label{ilw.4}
v(L,\delta,k):=\displaystyle\frac{2\delta}{L}\cdot\displaystyle\frac{K(k)}{K(k')}<1.
\end{equation}

Our next step is to present a convenient formula for the solution $u$. Consider the parameters $B$ and $w$ satisfying condition in (\ref{ilw.25.2.2}) and (\ref{ilw.25.2.1}), respectively, then by using formula 16.43.3-\cite {abramo} in (\ref{ilw.25.3}) one has
\begin{eqnarray}\label{ilw.5.1}
u(x,t)&=&\displaystyle\frac{iK(k)p}{\pi}
\displaystyle\left[Z\displaystyle\left(\displaystyle\frac{K(k)}{\pi}(z-ip\delta);k\right)
-Z\displaystyle\left(\displaystyle\frac{K(k)}{\pi}(z+ip\delta);k\right)\right]\\
&=&\displaystyle\frac{2K(k)i}{L}\displaystyle\left[Z\displaystyle\left(\displaystyle\frac{2K(k)}{L}(x-ct-i\delta);k\right)-Z\displaystyle\left(\displaystyle\frac{2K(k)}{L}(x-ct+i\delta);k\right)\right],
\end{eqnarray}
para $c:=-\displaystyle\frac{w}{p}$.  Therefore, identity
$(\ref{ilw.5.1})$ determines a class of $L-$periodic functions which
solves the ILW equation (\ref{ILW}) with speed-wave $c$. Here, $Z$
represents the periodic Jacobi Zeta Function (see section 2 above).

Next, we will determined an expression for  $c$.  Indeed, from the analysis above we obtain that
\begin{eqnarray*}c:=-\displaystyle\frac{w}{p}=\displaystyle\frac{1}{\delta}-ip\cdot \displaystyle\frac{\theta_1'(ip\delta,q(k))}{\theta_1(ip\delta,q(k))}=\displaystyle\frac{1}{\delta}-\displaystyle\frac{2\pi
i}{L}\cdot\displaystyle\frac{\theta_1'\displaystyle\left(\displaystyle\frac{2\pi\delta
i}{L},q(k)\right)}{\theta_1\displaystyle\left(\displaystyle\frac{2\pi\delta
i}{L},q(k)\right)}\cdot\end{eqnarray*} Thus, if we use formula
16.34.1 in \cite{abramo}, we get
\begin{eqnarray}\label{ilw.26}
c=\displaystyle\frac{1}{\delta}-\displaystyle\frac{4iK(k)}{L}\cdot
\displaystyle\left[Z\displaystyle\left(\displaystyle\frac{4i\delta
K(k)}{L};k\right)+\displaystyle\frac{\textrm{cn}\displaystyle\left(\displaystyle\frac{4i\delta
K(k)}{L};k\right)\cdot
\textrm{dn}\displaystyle\left(\displaystyle\frac{4i\delta
K(k)}{L};k\right)}{\textrm{sn}\displaystyle\left(\displaystyle\frac{4i\delta
K(k)}{L};k\right)}\right],
\end{eqnarray}
with $sn, cn, dn$ denoting the Jacobi elliptic functions {\it
snoidal, cnoida} and {\it dnoidal}, respectively
(see section 2 above). Hence, for $\xi=x-ct$ in (2.10) we
obtain the periodic traveling wave solution $\varphi_c$  in
(\ref{solilw}) for the ILW equation. Moreover, by construction one
has that $\varphi_c\in H_0$.

Next, by using formula 143.01 in \cite{byrd},  we can rewrite the profile $\varphi_c$ in function of the Jacobi elliptic functions snoidal, cnoida, and  dnoidal:
\begin{eqnarray}\label{ilw.30}
&&\varphi_c(x)=-\displaystyle\frac{4K(k)}{L}\cdot
Z\displaystyle\left(\displaystyle\frac{2K(k)\delta}{L};k'\right)-\displaystyle\frac{4\delta\pi}{L^2}\cdot\displaystyle\frac{K(k)}{K(k')}\\
\nonumber\\
&&+\displaystyle\frac{4K(k)}{L}\cdot\displaystyle\frac{\textrm{dn}^2\displaystyle\left(\displaystyle\frac{2K(k)x}{L};k\right)\cdot
\textrm{cn}\displaystyle\left(\displaystyle\frac{2K(k)\delta}{L};k'\right)
\cdot
\textrm{sn}\displaystyle\left(\displaystyle\frac{2K(k)\delta}{L};k'\right)
\cdot
\textrm{dn}\displaystyle\left(\displaystyle\frac{2K(k)\delta}{L};k'\right)}{1-
\textrm{dn}^2\displaystyle\left(\displaystyle\frac{2K(k)x}{L};k\right)\cdot
\textrm{sn}^2\displaystyle\left(\displaystyle\frac{2K(k)\delta}{L};k'\right)}.\nonumber
\end{eqnarray}
Figure 1 below, it shows the profile $ \varphi_c$ with some specific parameters of $L, \delta$ and $k$.

\begin{figure}[!htb]
\includegraphics[width=5cm,height=5cm]{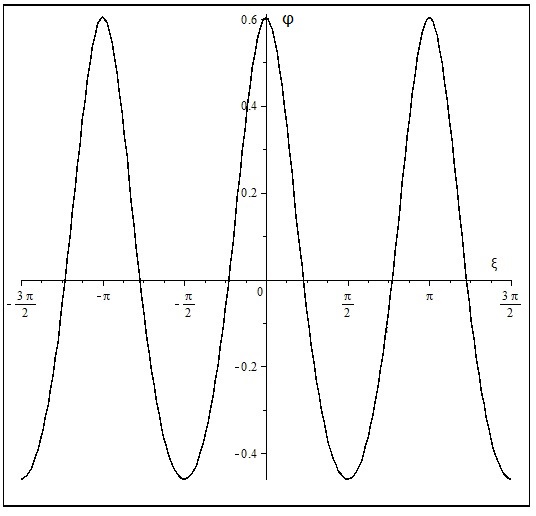}
\caption{\small Function $\varphi_{c}$ in
(\ref{ilw.30}) with $L=\pi$, $\delta=1$ and $k=0.5$.}
\end{figure}
Moreover, by using formulas 143.02, 161.01 and 120.02 in
\cite{byrd} at the identity (\ref{ilw.26}) one arrives to the  convenient formula for $c=c(k)$,
\begin{equation}\label{ilw.29}\begin{array}{llll}c&=&\displaystyle\frac{1}{\delta}-\displaystyle\frac{8\pi\delta K(k)}{L^2K(k')}-\displaystyle\frac{4K(k)}{L}\cdot Z\displaystyle\left(\displaystyle\frac{4\delta K(k)}{L};k'\right)\\
\\
&-&\displaystyle\frac{4K(k)}{L}\cdot\displaystyle\frac{\textrm{cn}\displaystyle\left(\displaystyle\frac{4\delta
K(k)}{L};k'\right)\cdot
\textrm{dn}\displaystyle\left(\displaystyle\frac{4\delta
K(k)}{L};k'\right)}{\textrm{sn}\displaystyle\left(\displaystyle\frac{4\delta
K(k)}{L};k'\right)}.\end{array}\end{equation}

Lastly, it follows immediate from  condition (\ref{ilw.4}) that for
$L$ and $\delta$ fixed there is an interval $(0, k_1)\subset (0,1)$,
with $k_1=k_1(L,\delta)$, such that $v(L, \delta, k)<1$ for all
$k\in (0, k_1)$. Therefore, we have the following existence result
of periodic traveling wave for the ILW equation by depending of the elliptic modulus $k$.

\begin{teo}\label{exis} For $L$ and $\delta$ fixed, there is $k_1 \in (0,1)$ such that for $c=c(k)$ defined in (\ref{ilw.29}) we have that the following smooth mapping
\begin{equation}\label{curve}
k\in (0, k_1)\to \varphi_{c(k)}\in H_{per}^n([0,L])\cap H_0\qquad n\in \mathbb N,
\end{equation}
it is well defined. Moreover, for every $k\in (0, k_1)$ we obtain that $\varphi_{k}=\varphi_{c(k)}$ satisfies (\ref{travkdv}) with $A_c=A(k)=\frac{1}{L}\int_0^L \varphi_{k}^2(x)dx$.
\end{teo}

In our analysis of linear and orbital stability of the profile $\varphi_{c(k)}$ in sections 5 and 6 below, we need to determine the sign of the derivate $\frac{d}{dk} c(k)$. For arbitrary values of $L$ and $\delta$ this calculation becomes a challenge. By making many numerical simulations with fixed values of $L$ and $\delta$ we obtain that $c=c(k)$ will always represent  a strictly increasing function on the specific interval $(0, k_1)$, and so we can assure the property $(P0)$. For instance,  the specific case of $L=\pi$ and $\delta=1$ we obtain the following plots for the function $c(k)$ and its derivate $c'(k)$, respectively,
\begin{figure}[!htb]
\begin{minipage}[b]{0.42\linewidth}
\begin{center}
\includegraphics[width=5.0cm,height=5.0cm]{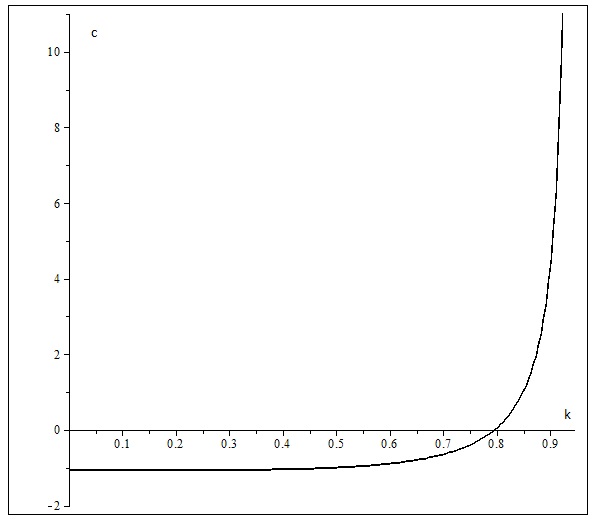}
\end{center}
\end{minipage} \hfill
\begin{minipage}[b]{0.42\linewidth}
\begin{center}
\includegraphics[width=5.0cm,height=5.0cm]{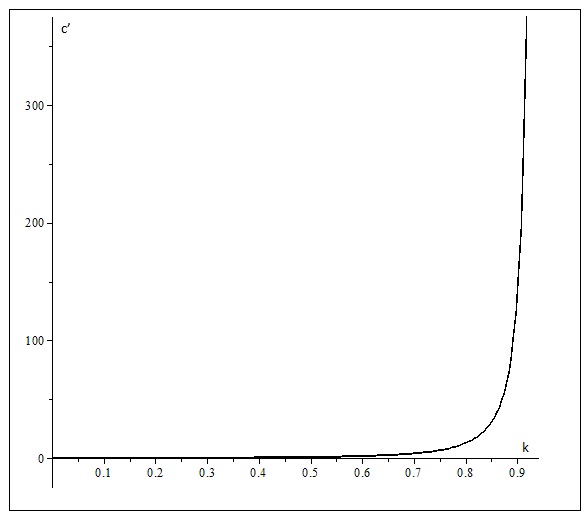}
\end{center}
\end{minipage}
\end{figure}\


Moreover, from the formula in (\ref{ilw.29}) and some numerical simulations, we obtain immediately  that  $k_1$ in Theorem \ref{exis}  has the approximation $k_1\approx 0,944085037$, and  for $k\in (0, k_1)$ we have the basic condition in  (\ref{ilw.4}), $v(\pi, 1, k)<1$, and
$$
c(0)=\displaystyle\lim_{k\rightarrow 0^{+}}c(k)\approx -1.07462944,\;\; \text{and}\;\;
\displaystyle\lim_{k\rightarrow k_1^{-}}c(k)= +\infty.
$$
We note that, there is a unique
$k_0\approx 0.795178532$ such that
\begin{eqnarray}\label{ilw.34.2}
c(k_0)=0\ \ \textrm{and}\ \ c(k)>0,\ \mbox{for all}\ k\in
(k_0,k_1),
\end{eqnarray}
therefore, the velocity $c$ is  negative on the interval $(0, k_0)$.

The simulations for the cases $L=\pi$ and differentes values of $\delta$, showed a similar behavior of the functions $c(k)$ and $c'(k)$ as showed above.

\section{Spectral Analysis}
\setcounter{equation}{0}
    \setcounter{defi}{0}
    \setcounter{teo}{0}
    \setcounter{lema}{0}
    \setcounter{prop}{0}
    \setcounter{coro}{0}

In this section, we start the analysis of the spectral problem
(\ref{vlinear2}) with $\mathcal{L}:=\mathcal{L}_{c,\delta}$ defined
in (\ref{linop}). The main idea for this study will be determine two
specific spectral  properties for $\mathcal L$, namely, that the
kernel is one-dimensional with $ker(\mathcal L)=[\frac{d}{dx}
\varphi_c]$ and the existence of a unique negative eigenvalue which
is simple.  Since the operator $\mathcal{L}$ is non-local this
analysis is not immediate.  In this point we will apply the
theory of Angulo and Natali put forward in  \cite{AN1} for studying the
stability of periodic traveling waves for the nonlinear dispersive  model
(\ref{equakawa}). The initial obstacle for applying Angulo and Natali's
approach is that the periodic traveling wave profile $\psi$
related to the equation (\ref{equakawa}) needs to be positive and satisfying the equation
$$
\mathcal M \psi+c\psi-\psi^2=0.
$$
Moreover, the wave speed $\varsigma$ needs to satisfy $\varsigma> -\inf_{r\in
\mathbb R} \theta (r)$ in order to determine that $\mathcal M +\varsigma$ is a positive
operator. In our analysis above (section 2), the traveling wave profile of $\varphi_c$ in (\ref{solilw})
has mean zero and the constant $A_c$ in (\ref{travkdv}) is not zero.
In order to overcome this difficulty, we shall use that the ILW equation is
invariant  by the Galilean transformation
$$
v(x,t)=u(x+2\gamma t,
t)-\gamma,
$$
for $\gamma$ being a real arbitrary value. The second obstacle is to determine the required spectral properties associated with the linearized operator $\mathcal{L}$ for arbitrary values of $L$ and $\delta$. So, by convenience in the exposition we shall restrict on a couple of specific values for $L$ and $\delta$, $L=\pi$ and $\delta=1$, respectively. However, numerical simulations enable us to conclude  that for other arbitrary values of $L$ and $\delta$ our results  remain valid.

In what follows, we establish some preliminaries definitions
and results due to Angulo and Natali's in \cite{natali}.

\begin{defi}\label{d.ilw.2} We say that a sequence $\alpha=(\alpha_n)_{n\in\mathbb{Z}}\subseteq
\mathbb{R}$ is in the class $PF(2)$
discrete if \\
\begin{enumerate}
\item[i)] $\alpha_n>0$,  for all $n\in\mathbb{Z}$,
\item[ii)] $\alpha_{n_1-m_1}\alpha_{n_2-m_2}-\alpha_{n_1-m_2}\alpha_{n_2-m_1}\geq 0$, for $n_1<n_2$ and $m_1<m_2$,
\item[iii)] $\alpha_{n_1-m_1}\alpha_{n_2-m_2}-\alpha_{n_1-m_2}\alpha_{n_2-m_1}> 0$,  if $n_1<n_2$, $m_1<m_2, n_2>m_1$, and $n_1<m_2$.
\end{enumerate}
\end{defi}
\indent The definition above is a particular case of the continuous
ones which appears in \cite{albert1} (see also Karlin \cite{Kar}), namely, we say that a function
$g:\mathbb R\to \mathbb R$
is in $PF(2)$-continuous if,
\begin{enumerate}
\item[i)] $g(x)>0$, for all $x\in\mathbb{R}$,
\item[ii)] $g(x_1-y_1)g(x_2-y_2)-g(x_1-y_2)g(x_2-y_1)\geq0$, for
$x_1<x_2$ and $\ y_1<y_2$,
\item[iii)] strict inequality holds in {\rm (ii)} whenever the intervals $(x_1, x_2)$ and $(y_1, y_2)$
intersect.\\
\end{enumerate}

An sufficient condition for $g$ belongs to $PF(2)$-continuous is for $g$ to be logarithmically concave, namely,
$$
\frac{d^2}{dx^2} log [g(x)]<0,\qquad x\neq 0.
$$
As an example of $PF(2)$-continuous functions, we have the profile
$Q_0(x)=\mbox{sech}^p(x)$, for $p>0$, and for $0<\nu<\mu$
$$
Q(x)=\frac{sinh(\nu x)}{sinh(\mu x)}.
$$
Hence, the sequences $(Q_0(n))_{n\in\mathbb{Z}}$ and $(Q(n))_{n\in\mathbb{Z}}$ belong to the class $PF(2)$
discrete.

The main theorem in  \cite{AN1} is the following

\begin{teo}\label{teonatali}
Suppose that $\psi_{\varsigma}$ is an even positive  solution of
(\ref{travkdv}) with $A\equiv0$, namely,
$$
\mathcal M \psi_{\varsigma}+\varsigma\psi_{\varsigma}-\psi_{\varsigma}^2=0,
$$
such that $\{\widehat{\psi_{\varsigma}}(n)\}_{n\in \mathbb Z}\in PF(2)$
discrete. Then the self-adjoint operator $\mathcal{L}_{\varsigma}:=\mathcal{M}+\varsigma -2\psi_{\varsigma}$
possesses only one negative eigenvalue which is simple and zero is a
simple eigenvalue with eigenfunction $\frac{d}{dx}\psi_{\varsigma}$. Moreover,
its spectrum is bounded away from zero.
\end{teo}



%

Our focus in the following  is to apply Theorem \ref{teonatali} in order to prove
our main result associated to the linear operator $\mathcal{L}$
in $(\ref{linop})$.

\begin{teo}\label{ilw.spec} Let $L=\pi$ and $\delta=1$ and consider $k\in (0,k_1)$, with $k_1$ defined by Theorem \ref{exis}. Then for $\varphi_c$ defined in (\ref{solilw}) with $c=c(k)$, we have that $\mathcal{L}$ in $(\ref{linop})$ is a self-adjoint operator with a discrete spectrum and satisfying
$\ker(\mathcal{L})=[\frac{d}{dx}\varphi_{c}]$. In addition,
$\mathcal{L}$ possess a unique negative eigenvalue which simple and
the remainder of the spectrum is constituted by isolated real
numbers which are bounded away from zero.
\end{teo}
{\proof  Initially, from the specific form of $\mathcal{L}$ we obtain from classical perturbation theory and spectral theory that $\mathcal{L}$ is a  self-adjoint operator with a discrete spectrum (see \cite{natali}).

Now,  in order to simplify the notation, we denote
\begin{eqnarray*}N(k):=\displaystyle\int_0^L\varphi_{c(k)}^2(x)\;dx, \qquad
R(k):=\displaystyle\frac{N(k)}{L},\end{eqnarray*}

\begin{eqnarray}\label{ilw.35} m_1:=\displaystyle\frac{4K(k)}{L}\cdot\textrm{cn}\displaystyle\left(\displaystyle\frac{2K(k)\delta}{L};k'\right) \cdot \textrm{sn}\displaystyle\left(\displaystyle\frac{2K(k)\delta}{L};k'\right) \cdot
\textrm{dn}\displaystyle\left(\displaystyle\frac{2K(k)\delta}{L};k'\right),\end{eqnarray}
\begin{eqnarray}\label{ilw.36} m_2:=\textrm{sn}^2\displaystyle\left(\displaystyle\frac{2K(k)\delta}{L};k'\right),
\end{eqnarray}
\begin{eqnarray}\label{ilw.38}m_3:=-\displaystyle\frac{4K(k)}{L}\cdot
Z\displaystyle\left(\displaystyle\frac{2K(k)\delta}{L};k'\right)-\displaystyle\frac{4\delta\pi}{L^2}\cdot\displaystyle\frac{K(k)}{K(k')},\
\ \textrm{and}\ \  m_4:=\displaystyle\frac{2K(k)}{L}.\end{eqnarray}

In the following analysis we will leave the parameters $L$ and $\delta$ fixed, but arbitrary. Thus,  from  (\ref{ilw.30}), (\ref{ilw.35}), (\ref{ilw.36}) and (\ref{ilw.38})
we get the expression
\begin{eqnarray}\label{ilw.39}\varphi_c(x)=m_1\cdot\displaystyle\frac{\textrm{dn}^2\displaystyle\left(m_4\cdot x;k\right)}{1-
m_2\cdot\textrm{dn}^2\displaystyle\left(m_4\cdot x;k\right)}+m_3
\end{eqnarray}
and, consequently,
\begin{eqnarray}\label{ilw.N}N(k)&=&m_1^2\cdot\displaystyle\int_0^L\displaystyle\frac{\textrm{dn}^4\displaystyle\left(m_4\cdot x;k\right)}{\displaystyle\left[1-
m_2\cdot\textrm{dn}^2\displaystyle\left(m_4\cdot x;k\right)\right]^2}\
dx\nonumber\\
\\
&+&2m_1m_3\cdot \displaystyle\int_0^L
\displaystyle\frac{\textrm{dn}^2\displaystyle\left(m_4\cdot x;k\right)}{1-
m_2\cdot\textrm{dn}^2\displaystyle\left(m_4\cdot x;k\right)} \
dx+Lm_3^2.\nonumber\end{eqnarray}

Next, by using formula 410.04 in \cite{byrd} we deduce
\begin{eqnarray}\label{ilw.41}\displaystyle\int_0^L
\displaystyle\frac{\textrm{dn}^2\displaystyle\left(m_4\cdot x;k\right)}{1-
m_2\cdot\textrm{dn}^2\displaystyle\left(m_4\cdot x;k\right)} \
dx&=&\displaystyle\frac{1}{m_4}\cdot\displaystyle\int_0^{m_4\cdot
L}\displaystyle\frac{\textrm{dn}^2\displaystyle\left(\zeta;k\right)}{1-
m_2\cdot\textrm{dn}^2\displaystyle\left(\zeta;k\right)} \
d\zeta\nonumber\\
\nonumber\\
&=&\displaystyle\frac{1}{m_4}\cdot\displaystyle\int_0^{2K(k)}\displaystyle\frac{\textrm{dn}^2\displaystyle\left(\zeta;k\right)}{1-m_2+
m_2\cdot k^2\cdot\textrm{sn}^2\displaystyle\left(\zeta;k\right)} \
d\zeta\nonumber\\
\\
&=&\displaystyle\frac{2}{m_4\cdot(1-m_2)}\cdot\displaystyle\int_0^{K(k)}\displaystyle\frac{\textrm{dn}^2\displaystyle\left(\zeta;k\right)}{1-\alpha^2\cdot\textrm{sn}^2\displaystyle\left(\zeta;k\right)}
\ d\zeta\nonumber\\
\nonumber\\
&=&\displaystyle\frac{2}{m_4\cdot(1-m_2)}\cdot\displaystyle\left[\displaystyle\frac{\pi\cdot(k^2-\alpha^2)\cdot
\Lambda_0(\psi,k)}{2\displaystyle\sqrt{\alpha^2\cdot(1-\alpha^2)\cdot
(\alpha^2-k^2)}}\right],\nonumber
\end{eqnarray}
where
\begin{equation}\label{ilw.42}\alpha^2=-\displaystyle\frac{m_2\cdot k^2}{1-m_2}<0,\
\ m_2\neq 1,\ \
\psi=\sin^{-1}\displaystyle\left(\displaystyle\sqrt{\displaystyle\frac{\alpha^2}{\alpha^2-k^2}}\right),\end{equation}
and $\Lambda_0$ indicates the Lambda Heuman function defined
by\begin{equation}\label{ilw.42.1}\Lambda_0(\psi,k)=\displaystyle\frac{2}{\pi}\cdot\displaystyle\left[E(k)\cdot
F(\psi,k')+K(k)\cdot E(\psi,k')-K(k)\cdot
F(\psi,k')\right],\end{equation} where
\begin{equation}\label{ilw.42.2}E(k)=\displaystyle\int_0^1\displaystyle\sqrt{\displaystyle\frac{1-k^2t^2}{1-t^2}}\ dt,\ \ E(\psi,k')=\displaystyle\int_0^{\psi}\displaystyle\sqrt{1-(1-k^2)\sin^2(\theta)}\ d\theta\end{equation}
and
\begin{equation}\label{ilw.42.3}F(\psi,k')=\displaystyle\int_0^{\psi}\displaystyle\frac{d\theta}{\displaystyle\sqrt{1-(1-k^2)\sin^2(\theta)}}.\end{equation}
Therefore, formula 410.08 in \cite{byrd} enables us to conclude
\begin{eqnarray}\label{ilw.43}\displaystyle\int_0^L
\displaystyle\frac{\textrm{dn}^4\displaystyle\left(m_4\cdot x;k\right)}{\displaystyle\left[1-
m_2\cdot\textrm{dn}^2\displaystyle\left(m_4\cdot x;k\right)\right]^2}
\
dx&=&\displaystyle\frac{1}{m_4}\cdot\displaystyle\int_0^{m_4\cdot
L}\displaystyle\frac{\displaystyle\left[\textrm{dn}^2\displaystyle\left(\zeta;k\right)\right]^2}{\displaystyle\left[1-
m_2\cdot\textrm{dn}^2\displaystyle\left(\zeta;k\right)\right]^2} \
d\zeta\nonumber\\
\nonumber\\
&=&\displaystyle\frac{1}{m_4}\cdot\displaystyle\int_0^{2K(k)}\displaystyle\frac{\displaystyle\left[1-k^2\cdot\textrm{sn}^2\displaystyle\left(\zeta;k\right)\right]^2}{\displaystyle\left[1-m_2+
m_2\cdot
k^2\cdot\textrm{sn}^2\displaystyle\left(\zeta;k\right)\right]^2}
\ d\zeta\nonumber\\
\nonumber\\
&=&\displaystyle\frac{2}{m_4\cdot(1-m_2)^2}\cdot\displaystyle\int_0^{K(k)}\displaystyle\frac{\displaystyle\left[1-k^2\cdot
\textrm{sn}^2\displaystyle\left(\zeta;k\right)\right]^2}{\displaystyle\left[1-\alpha^2\cdot\textrm{sn}^2\displaystyle\left(\zeta;k\right)\right]^2}
\ d\zeta\nonumber\\
\\
&=&\displaystyle\frac{2}{m_4\cdot(1-m_2)^2}\cdot\displaystyle\frac{1}{\alpha^4}\cdot\displaystyle\left[k^4\cdot
K(k)\right.\nonumber\\
\nonumber\\
&+&\displaystyle\left.2\cdot k^2\cdot (\alpha^2-k^2)\cdot
\Pi(\alpha^2,k)+(\alpha^2-k^2)^2\cdot V_2\right],\nonumber
\end{eqnarray}
where \begin{eqnarray}\label{ilw.44}\Pi(\alpha^2,k)=\displaystyle\frac{k^2\cdot K(k)}{k^2-\alpha^2}-\displaystyle\frac{\pi\cdot \alpha^2\cdot \Lambda_0(\psi,k)}{2\displaystyle\sqrt{\alpha^2\cdot (1-\alpha^2)\cdot
(\alpha^2-k^2)}}\end{eqnarray} and
\begin{eqnarray}\label{ilw.45}V_2&=&\displaystyle\frac{1}{2\cdot (\alpha^2-1)\cdot (k^2-\alpha^2)}\cdot\displaystyle\left\{\displaystyle\frac{\displaystyle\left[2\cdot k^4\cdot\alpha^2-2\cdot k^4+\alpha^4\cdot (1-k^2)\right]\cdot
K(k)}{k^2-\alpha^2}\right.\nonumber\\
\\
&+&\displaystyle\left.\alpha^2\cdot
E(k)-\displaystyle\frac{\pi\cdot\displaystyle\left(2\cdot\alpha^2\cdot
k^2+2\cdot\alpha^2-\alpha^4-3\cdot k^2\right)\cdot \alpha^2\cdot
\Lambda_0(\psi,k)}{2\displaystyle\sqrt{\alpha^2\cdot
(1-\alpha^2)\cdot (\alpha^2-k^2)}}\right\}.\nonumber
\end{eqnarray}
Statements
(\ref{ilw.35})-(\ref{ilw.45}), give us
\begin{eqnarray}\label{ilw.46}N(k)&=&\displaystyle\frac{2\cdot m_1^2}{m_4\cdot(1-m_2)^2}\cdot\displaystyle\frac{1}{\alpha^4}\cdot\displaystyle\left[k^4\cdot
K(k)\right.\nonumber\\
\nonumber\\
&+&\displaystyle\left.2\cdot k^2\cdot (\alpha^2-k^2)\cdot
\Pi(\alpha^2,k)+(\alpha^2-k^2)^2\cdot V_2\right]\\
\nonumber\\
&+&\displaystyle\frac{2\cdot m_1\cdot
m_3}{m_4\cdot(1-m_2)}\cdot\displaystyle\left[\displaystyle\frac{\pi\cdot(k^2-\alpha^2)\cdot
\Lambda_0(\psi,k)}{\displaystyle\sqrt{\alpha^2\cdot(1-\alpha^2)\cdot
(\alpha^2-k^2)}}\right]+Lm_3^2.\nonumber
\end{eqnarray}\

Next, by considering the specific values of  $L=\pi$, $\delta=1$, we obtain for all $k\in (0,k_1)$ ($k_1\approx 0,944085037$), the existence of $a=a(k)>0$
such that
\begin{eqnarray}\label{ilw.51}a^2+ca-R=0.\end{eqnarray} In fact, one has
\begin{eqnarray}\label{ilw.52}a=\displaystyle\frac{-c+\displaystyle\sqrt{c^2+4R}}{2}.
\end{eqnarray}
Moreover, by using that
$$
\min_{x\in
[0,L]}\varphi_c(x)=\varphi_c\displaystyle\left(\displaystyle\frac{L}{2}\right),
$$
we find via numerical simulations (see Figure \ref{f.ilw.8} below)
that
\begin{eqnarray}\label{ilw.52.1}
a(k)> -\varphi_{c(k)}\displaystyle\left(\displaystyle\frac{L}{2}\right),\qquad \text{for all}\;\; k\in (0,k_1).
\end{eqnarray}

\begin{figure}[!h]
\includegraphics[width=7.0cm,height=6.0cm]{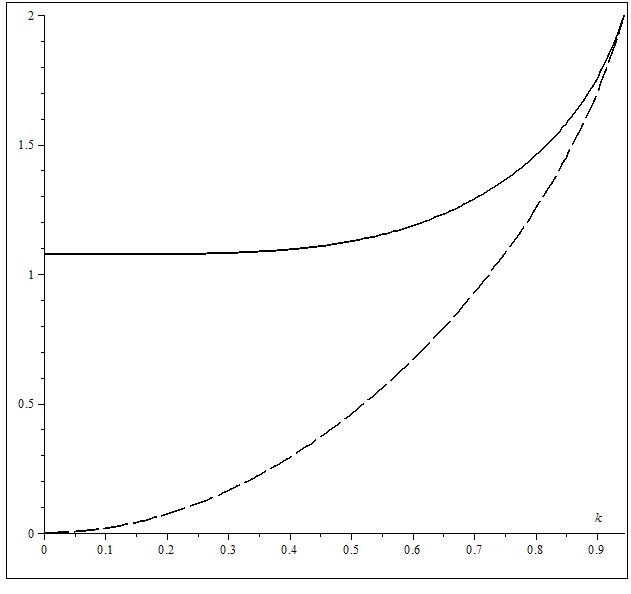}
\caption{\label{f.ilw.8}\small Consider $L=\pi$ and $\delta=1$. The
continuous line gives us the behavior of the function $a=a(k)$ for
$k\in (0,k_1)$. On the other hand, the dashed line shows us the
behavior of $-\varphi_{c(k)}( L/2)$ for $k\in(0,k_1)$.}
\end{figure}

Next, let us define $
\varsigma=\varsigma(k)$ by
$$
\varsigma:=c+2a=\displaystyle\sqrt{c^2+4R}>0
$$
and we consider the
translation function $\phi_{\varsigma}:=a+\varphi_c$. By using
(\ref{ilw.52.1}), we conclude $\phi_{\varsigma}>0$. Moreover, since
$\varphi_c$ is an even $L-$periodic function one has that
$\phi_{\varsigma}$ is also an even $L-$periodic function. Now, we claim
that $\phi_{\varsigma}$ satisfies equation (\ref{travkdv}) with
$A\equiv0$. Indeed, since
$\mathcal{M}_{\delta}(\varphi_c+\beta)=\mathcal{M}_{\delta}\varphi_c$,
for all $\beta\in \mathbb R$, it follows from (\ref{travkdv}) and
(\ref{ilw.51}) that
\begin{eqnarray*}-\mathcal{M}_{\delta}\phi_{\varsigma}-\varsigma
\phi_{\varsigma}+\phi_{\varsigma}^2&=&-\mathcal{M}_{\delta}\phi_{\varsigma}-c
\phi_{\varsigma}-2a\phi_{\varsigma}+\phi_{\varsigma}^2\nonumber\\
&=&-\mathcal{M}_{\delta}\varphi_c-c
(\varphi_c+a)-2a(\varphi_c+a)+(\varphi_c+a)^2\nonumber\\
&=&-\mathcal{M}_{\delta}\varphi_c-c\varphi_c+\varphi_c^2-(ca+a^2)=0.\nonumber\\
\end{eqnarray*}

In what follows, we will verify that for all $k\in (0,k_1)$,
$\{\widehat{\phi_{\varsigma}}(n)\}_{n\in \mathbb Z}\in PF(2)$ discrete.
We recall that such values of $k$ satisfy the analytic condition (\ref{ilw.4}). Applying  formula 905.01
of \cite{byrd} in (\ref{solilw}), we obtain
\begin{eqnarray}\label{ilw.53}\varphi_c(x)&=&\displaystyle\frac{2\pi
i}{L}\displaystyle\left[\displaystyle\sum_{m=1}^{+\infty}\displaystyle\frac{\sin\displaystyle\left(\displaystyle\frac{2m\pi}{L}\displaystyle\left(x-i\delta\right)\right)}{\sinh\displaystyle\left(\displaystyle\frac{m\pi
K(k')}{K(k)}\right)}-\displaystyle\sum_{m=1}^{+\infty}\displaystyle\frac{\sin\displaystyle\left(\displaystyle\frac{2m\pi}{L}\displaystyle\left(x+i\delta\right)\right)}{\sinh\displaystyle\left(\displaystyle\frac{m\pi
K(k')}{K(k)}\right)}\right]\\
\nonumber\\
&=&\displaystyle\frac{4\pi}{L}\displaystyle\sum_{m=1}^{+\infty}\displaystyle\frac{\sinh\displaystyle\left(\displaystyle\frac{2m\pi\delta}{L}\right)}{\sinh\displaystyle\left(\displaystyle\frac{m\pi
K(k')}{K(k)}\right)}\cdot\cos\displaystyle\left(\displaystyle\frac{2m\pi x}{L}\right),\nonumber
\end{eqnarray} that is,
\begin{eqnarray}\label{ilw.54}\phi_{\varsigma}(x)=a+\displaystyle\frac{4\pi}{L}\displaystyle\sum_{m=1}^{+\infty}\displaystyle\frac{\sinh\displaystyle\left(\displaystyle\frac{2m\pi\delta}{L}\right)}{\sinh\displaystyle\left(\displaystyle\frac{m\pi
K(k')}{K(k)}\right)}\cdot\cos\displaystyle\left(\displaystyle\frac{2m\pi x}{L}\right).
\end{eqnarray}
So, the periodic Fourier transform related to the function $\phi_{\varsigma}$ is expressed by
$\widehat{\phi_{\varsigma}}(0)=a$ and
\begin{equation}\label{ilw.55}\widehat{\phi_{\varsigma}}(m)=\displaystyle\frac{2\pi}{L}\cdot\displaystyle\frac{\sinh\displaystyle\left(\displaystyle\frac{2m\pi\delta}{L}\right)}{\sinh\displaystyle\left(\displaystyle\frac{m\pi
K(k')}{K(k)}\right)},\;\;\;\;\ \mbox{for all}\ m\in\mathbb{Z}-\{0\}.
\end{equation}

Letting $$\nu:=\displaystyle\frac{2\pi\delta}{L}\ \ \ \textrm{and}\
\ \ \mu:=\displaystyle\frac{\pi K(k')}{K(k)},
$$
we obtain from (\ref{ilw.4}) immediately that $0<\nu<\mu.$ On the other hand, by considering
\begin{eqnarray}\label{ilw.55.1}Q(x):=\displaystyle\frac{\sinh(\nu x)}{\sinh(\mu
x)},\qquad x\neq 0,
\end{eqnarray}
we see that
\begin{eqnarray}\label{ilw.56.1}\displaystyle\frac{d^2}{dx^2}\displaystyle\left[\log(Q(x))\right]<0,\qquad
\forall\ x\neq 0.\end{eqnarray}
Therefore,  we obtain that $Q\in PF(2)$-continuous (see \cite{albert1}). In addition, we obtain  the following specific calculation to be used below,
\begin{eqnarray}\label{ilw.56.2}\displaystyle\lim_{x\rightarrow 0}\displaystyle\frac{2\pi}{L}\cdot\displaystyle\frac{\sinh\displaystyle\left(\displaystyle\frac{2\pi\delta x}{L}\right)}{\sinh\displaystyle\left(\displaystyle\frac{\pi K(k')x}{K(k)}\right)}=\displaystyle\frac{4\pi\delta
K(k)}{L^2K(k')}.\end{eqnarray}

Next, the  following picture show us that the function
$$
a(k)-\displaystyle\frac{2\pi}{L}\cdot v(L,\delta,k)\equiv a(k)-2\cdot
v(\pi,1,k)
$$
for $k\in(0,k_1)$, it is strictly positive mapping.

\begin{figure}[!h]\begin{center}
\includegraphics[width=6.0cm,height=6.0cm]{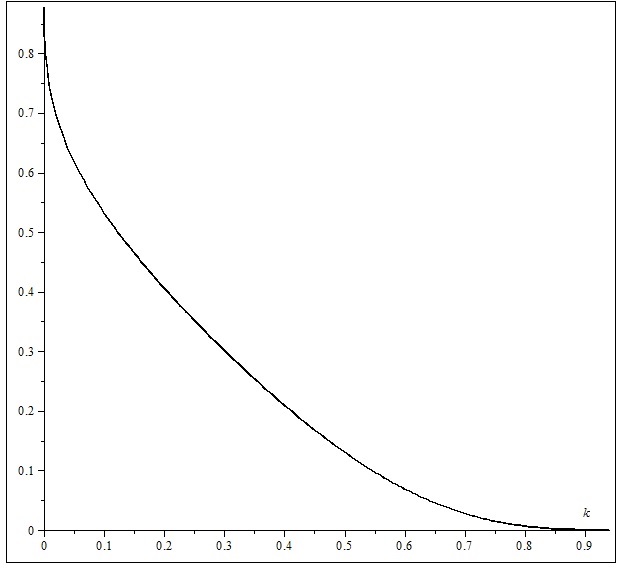}
\end{center}
\end{figure}
Therefore, we obtain for all $k\in (0,k_1)$ the relation
\begin{eqnarray}\label{ilw.57}a>\displaystyle\frac{4\pi\delta
K(k)}{L^2K(k')}=\displaystyle\frac{2\pi}{L}\cdot
v(L,\delta,k).
\end{eqnarray}
Hence, the statements (\ref{ilw.55.1})-(\ref{ilw.57}) allow us to define a smooth function
$\tau:\mathbb{R}\rightarrow\mathbb{R}$ such that
$$
\tau(x):=\displaystyle\frac{2\pi Q(x)}{L},\quad\quad \forall\ x\in
(-\infty,-1]\cup[1,+\infty)
$$
$\tau(0)=a$ and $\tau $ in $(-1, 1)$ such that $\tau\in PF(2)$ continuous. Therefore, we can conclude that
\begin{eqnarray*}
\{\widehat{\phi_{\varsigma}}(m)\}_{m\in\mathbb{Z}}\in PF(2) \ \textrm{discrete}.
\end{eqnarray*}

Hence, from Theorem \ref{teonatali} we obtain that the linear operator $\mathcal{L}_{\varsigma,\delta}=\mathcal{M}_{\delta}+\varsigma-2\phi_{\varsigma}$ admits exactly one negative eigenvalue which is simple and zero is also a simple eigenvalue whose correspondent eigenfunction is $\frac{d}{dx}\phi_{\varsigma}$. Lastly,  we analyze the operator $\mathcal L$. Indeed, since
\begin{eqnarray}
\label{ilw.58}\mathcal{L}_{\varsigma,\delta}=\mathcal{M}_{\delta}+\varsigma-2\phi_{\varsigma}=\mathcal{M}_{\delta}+(c+2a)-(2\varphi_c+2a)
=\mathcal{M}_{\delta}+c-2\varphi_c=\mathcal{L},
\end{eqnarray}
then we obtain
\begin{eqnarray}\label{ilw.59}
\ker(\mathcal{L})=\ker(\mathcal{L}_{\varsigma,\delta})=\Big [\frac{d}{dx}\phi_{\varsigma}\Big]=\Big[\frac{d}{dx}\varphi_{c}\Big],\;\;\;\;\text{and}\;\;\;\;n(\mathcal{L})=1.
\end{eqnarray}
This finishes the Theorem.

\begin{flushright} $\square$ \end{flushright}

 \begin{obs}\label{restri} To study the behaviour of the function $a=a(k)$ in $(\ref{ilw.52})$ in order to determine that $(\ref{ilw.52.1})$ holds for arbitrary values of  $L$ and $\delta$ will induce  enormous technical difficulties if we do not use numerical simulations for fixing values of $L$ and $\delta$. Maple 16 enable us to conclude that $(\ref{ilw.52.1})$ remains still valid for general values of $L$ and $\delta$ satisfying the analytic condition in (\ref{ilw.4}). As a consequence, the results in Theorem $\ref{ilw.spec}$ can be established for general values of $L$ and $\delta$.
 \end{obs}

\section{Linear Stability for the ILW-Equation}

\indent

In this section we establish our linear stability result for the mean zero traveling wave $\varphi_{c}$ in (\ref{solilw}). For the convenience of the reader we will give some definitions and specific  sufficient conditions for obtaining our  linear stability result (see \cite{DK} and \cite{haragus}).

We start our study by establishing some definitions associated to the operator $\partial_x \mathcal L|_{H_0}$, with $\mathcal L= \mathcal L_{c,\delta}$ in (\ref{linop}) and $H_0$  in (\ref{zero}).

\begin{defi}\label{K1}
We define,
\begin{enumerate}
\item $k_r$ as the number of positive real eigenvalues  (counting
multiplicities) of the operator $\partial_x \mathcal L|_{H_0}$.
\item  $k_c$ indicates the number of complex-valued eigenvalues with a positive real part (counting
multiplicities) of the operator $\partial_x \mathcal L|_{H_0}$.
\item For $B$ a linear operator with domain $D(B)$,  we define the linear operator $Im(B)u\equiv Im (Bu)$ for $u\in D(B)$.
\end{enumerate}
\end{defi}

We note immediately from the later Definition, that since $Im(\mathcal L)\equiv 0$ then  $k_c$ is an even
integer.  Next, for a self-adjoint operator $\mathcal{A}$, we denote by  $n(\langle
w,\mathcal{A}w\rangle)$  the dimension of the maximal subspace for
which $\langle w,\mathcal Aw\rangle<0$ (Morse index of $\mathcal A$). Also, let $\lambda$ be an
eigenvalue for $\partial_x \mathcal L$ and $E_{\lambda}$ its corresponding eigenspace. The
eigenvalue is said to have negative Krein signature if
$$k_i^{-}(\lambda):=n(\langle w,(\mathcal{L}\big|_{H_0})\big|_{E_{\lambda}}w\rangle)\geq1,$$
otherwise, if $k_i^{-}=0$, then the eigenvalue is said to have a positive Krein signature. If $\lambda$ is a geometrically and algebraically simple eigenvalue for $\partial_x \mathcal L$ with  eigenfunction $\psi_{\lambda}$ then $E_\lambda=[\psi_\lambda]$, and so
$$k_i^{-}(\lambda)=\left\{\begin{array}{llll}
0,\ \langle \psi_{\lambda},(\mathcal{L}\big|_{H_0})\psi_{\lambda}\rangle>0\\
1,\ \langle \psi_{\lambda},(\mathcal{L}\big|_{H_0})\psi_{\lambda}\rangle<0.\end{array}
\right.$$

The total Krein signature is given by
$k_i^{-}:=\sum _{\lambda\in i\mathbb{R}\backslash\{0\}}k_{i}^{-}(\lambda).$
Since ${\rm Im}(\mathcal{L})=0$ we obtain that $k_i^{-}$ is an even integer.\\

\begin{defi}\label{K2}
The Hamiltonian-Krein index associated to the operator $\partial_x \mathcal L$ is the following non-negative integer
$$
\mathcal{K}_{{\rm Ham}}=k_r+k_c+k_{i}^{-}.
$$
\end{defi}

Next, let us consider the quantity
\be\label{I}
\mathcal{I}=\langle \mathcal{L}^{-1}1,1\rangle.
\ee
 We also note that for any $f\in ker(\mathcal L)^\bot$ the quantity $ \langle \mathcal{L}^{-1}f,f\rangle$ is always independent of $h\in  \mathcal{L}^{-1}f$. Now, we denote by $\mathcal{D}$  the $2\times 2-$matrix given by
\be\label{Dmatrix}
\mathcal{D}=\frac{1}{\langle\mathcal{L}^{-1}1,1\rangle}\left[\begin{array}{llll}
\langle\mathcal{L}^{-1}\varphi_{c},\varphi_{c}\rangle & & \langle\mathcal{L}^{-1}\varphi_{c},1\rangle\\\\
\langle\mathcal{L}^{-1}\varphi_{c},1\rangle & & \langle\mathcal{L}^{-1}1,1\rangle\end{array}\right].
\ee
Then, from \cite{DK} and \cite{haragus} we have  the following results:
\begin{teo}\label{krein}
Suppose that $ker(\mathcal L)=[\frac{d}{dx} \varphi_{c}]$. If $\mathcal{I}\neq0$ and
$\mathcal{D}$ is non-singular we have for the eigenvalue problem in
$(\ref{modspecp1})$ the following relation
$$
\mathcal{K}_{{\rm Ham}}=n(\mathcal{L})-n(\mathcal{I})-n(\mathcal{D}).
$$
\end{teo}

We recall that $n(\mathcal{I})=0\Leftrightarrow\mathcal{I}>0$ and $n(\mathcal{I})=1\Leftrightarrow\mathcal{I}<0$. An immediate consequence of Theorem \ref{krein} is the following criterium of linear stability.

\begin{coro}\label{coroest}
Under the assumptions of Theorem $\ref{krein}$, if $k_c=k_r=k_{i}^{-}=0$ then the periodic wave $\varphi_{c}$ is linearly stable. In addition, if $\mathcal{K}_{{\rm Ham}}=1$ then  the refereed periodic wave is linearly unstable.
\end{coro}
\begin{proof}
The first part of the corollary is an immediate consequence of
Theorem 2.7 in \cite{DK} (see also \cite{haragus}). Now, if $\mathcal{K}_{{\rm
Ham}}=1$ we deduce that $k_r=1$ since $k_c$ and $k_{i}^{-}$ are even
nonnegative integers. Then,  the
spectral problem in $(\ref{modspecp1})$ has a positive eigenvalue which
able us to deduce the linear instability of the periodic wave $\varphi_{c}$.
\end{proof}

\indent Next we establish our linear stability result associated to the periodic traveling wave $\varphi_{c}$ in (\ref{ilw.30}). Since our study will be based on Theorem \ref{krein}, the value of $\mathcal{K}_{{\rm Ham}}$ must be calculated. From Theorem \ref{ilw.spec} we have that $n(\mathcal{L})=1$. Next will prove that $n(\mathcal{D})=1$ and $n(\mathcal{I})=0$ by considering the case of $c$ being positive by technical reasons. For obtaining these quantities we will need to calculate some expressions for $\mathcal{I}$ and $\det(\mathcal{D})$ in terms of the Jacobi elliptic functions. More explicitly, we will obtain (see propositions below) the following explicit formulas:
\begin{equation}\label{nI}
\mathcal{I}=\langle\mathcal{L}^{-1}1,1\rangle=\displaystyle\frac{L^2}{cL+2\displaystyle\frac{\partial}{\partial c}\displaystyle\left[\displaystyle\int_0^L\varphi^2_{c}(x)dx\right]},\end{equation}
and \be\label{detD}
\det(\mathcal{D})=-\frac{1}{2}\frac{1}{\mathcal{I}}\frac{\partial}{\partial
c}\int_0^L \varphi_c^2(x)dx.
\ee
Thus, we will prove that  $\frac{\partial}{\partial c}\int_0^L\varphi_c^2(x)dx>0$ and therefore
$\mathcal{I}>0$ and $\det(\mathcal{D})<0$. Hence, $n(\mathcal{I})=0$ and
$n(\mathcal{D})=1$. Therefore, from Theorem \ref{ilw.spec} and Theorem \ref{krein} we conclude that $\mathcal{K}_{\rm
Ham}=0$. Then, by Corollary $\ref{coroest}$ one has that  the periodic wave $\varphi_c$ is linearly stable. Formally, we have the following  linear stability result.

\begin{teo}\label{t.ilw.4.1} Consider $c>0$.  The periodic traveling waver $\varphi_{c}$ in $(\ref{ilw.30})$ is  linearly stable  for the ILW equation.
\end{teo}

The focus of the following propositions will be to show that $\mathcal{I}>0$ and $\det(\mathcal{D})<0$. We recall that for convenience in the exposition we are considering $L=\pi$ and $\delta=1$. We start by establishing the following main result.

\begin{prop}\label{form1}  For $c>0$ one has $\frac{\partial}{\partial c}\int_0^L\varphi_c^2(x)dx>0$. \end{prop}

\begin{proof} We start with the relation
\begin{eqnarray}\label{ilw.33}\displaystyle\frac{\partial}{\partial
c}\displaystyle\left[\displaystyle\int_0^L\varphi_c^2(x)dx\right]=\displaystyle\frac{dk}{dc}\cdot\displaystyle\frac{\partial}{\partial
k}\displaystyle\left[\displaystyle\int_0^L\varphi_{c(k)}^2(x)
dx\right]\equiv \displaystyle\frac{dk}{dc}\cdot N'(k).\end{eqnarray}
Thus, since  $c'(k)>0$, for all $k\in (k_0,k_1)$ (see (\ref{ilw.34.2})), we only need to
establish the sign of $N'(k)$. Before that, it makes necessary to handle with the quantity $N(k)$
in (\ref{ilw.N}) for obtaining  a convenient expression for our calculations. Indeed, from (\ref{ilw.53}) and Plancherel Theorem, we obtain
\begin{eqnarray}\label{ilw.33.2}N(k)=\displaystyle\int_0^L\varphi_{c(k)}^2(x) dx=L\displaystyle\sum_{m=-\infty}^{+\infty}|\widehat{\varphi_c}(m)|^2=\displaystyle\frac{8\pi^2}{L}\displaystyle\sum_{m=1}^{+\infty}\displaystyle\frac{\displaystyle\left[\sinh\displaystyle\left(\displaystyle\frac{2m\pi\delta}{L}\right)\right]^2}{\displaystyle\left[\sinh\displaystyle\left(\displaystyle\frac{m\pi
K(k')}{K(k)}\right)\right]^2},\ \
\end{eqnarray} for all $k\in(0,k_1)$.
So, one can take the first derivative with respect to $k\in
(k_0,k_1)$ in (\ref{ilw.33.2}) to deduce
\begin{eqnarray*}N'(k)=-\displaystyle\frac{16\pi^3}{L}\displaystyle\sum_{m=1}^{+\infty}\displaystyle\left\{\displaystyle\frac{m\cdot\displaystyle\left[\sinh\displaystyle\left(\displaystyle\frac{2m\pi\delta}{L}\right)\right]^2\cdot\displaystyle\left[\cosh\displaystyle\left(\displaystyle\frac{m\pi K(k')}{K(k)}\right)\right]\cdot\displaystyle\left[\displaystyle\frac{d}{dk}\displaystyle\left[\displaystyle\frac{K(k')}{K(k)}\right]\right]}{\displaystyle\left[\sinh\displaystyle\left(\displaystyle\frac{m\pi
K(k')}{K(k)}\right)\right]^3}\right\}.
\end{eqnarray*}

Since
\begin{eqnarray*}\displaystyle\frac{d}{dk}\displaystyle\left[\displaystyle\frac{K(k')}{K(k)}\right]=\displaystyle\frac{[E(k)-K(k)]\cdot K(k')+K(k)\cdot E(k')}{k(k^2-1)\cdot K(k)^2}<0,\quad\text{for all}\;\; k\in (0,1),\end{eqnarray*} we obtain immediately that
\begin{eqnarray}\label{ilw.34}N'(k)>0,\quad\text{for all}\;\; k\in (k_0,k_1).
\end{eqnarray}
This finishes the proof.
\end{proof}

\begin{obs}  By using the proof of Proposition \ref{form1} and the numerical calculations made in Section 2 (see (\ref{ilw.34.2})) we see that $N'(k)>0$ for every $k\in (0, k_1)-\{k_0\}$. So,  we have
\begin{equation}
\frac{d}{dc}\|\varphi_c\|^2>0, \quad\text{for every}\;\; c\neq 0
\end{equation}
  \end{obs}

Next we establish the formulas (\ref{nI}) and (\ref{detD}).

\begin{prop}\label{form2}  For every $c>0$ we obtain $\mathcal{I}>0$. In particular, $n(\mathcal{I})=0$.

\end{prop}

\begin{proof} Since $f\equiv 1\in
H_{per}^s([0,L])$, for all $s\geq0$, and $\mathcal{M}_{\delta}(1)=0$ we get
\begin{eqnarray}\label{ilw.15}
\mathcal{L}(1)=\mathcal{M}_{\delta}(1)+c-2\varphi_{c}=c-2\varphi_{c}.
\end{eqnarray}
Then, since $\ker(\mathcal{L})=[\frac{d}{dx}\varphi_c]$, $\frac{d}{dx}\varphi_c\perp
1$ and $\frac{d}{dx}\varphi_c\perp \varphi_c$, one has from
(\ref{ilw.15}) that
\begin{eqnarray}\label{ilw.15.1}1=c\mathcal{L}^{-1}1-2\mathcal{L}^{-1}\varphi_{c}.
\end{eqnarray}
Thus
$$
c\langle\mathcal{L}^{-1}1,1\rangle=\langle
1,1\rangle+2\langle\mathcal{L}^{-1}\varphi_{c},1\rangle.
$$
Now, since $c>0$, we get
\begin{eqnarray}\label{ilw.16}\langle\mathcal{L}^{-1}1,1\rangle=\displaystyle\frac{L}{c}
+\displaystyle\frac{2\langle\mathcal{L}^{-1}\varphi_{c},1\rangle}{c}.
\end{eqnarray}

Next, by differentiating  identity (\ref{travkdv}) with regard to $c$ we obtain
\begin{eqnarray}\label{ilw.17}
\mathcal{L}\displaystyle\left(\displaystyle\frac{\partial}{\partial c}\varphi_{c}\right)=-\varphi_{c}-\displaystyle\frac{1}{L}\frac{d}{dc}\|\varphi_{c}\|^2.
\end{eqnarray}
Then,  by applying the  operator $\mathcal{L}^{-1}$ at the equality (\ref{ilw.17}) we deduce
\begin{equation}\label{ilw.19}
\displaystyle\frac{\partial}{\partial c}\varphi_{c}=-\mathcal{L}^{-1}\varphi_{c}-\displaystyle\frac{1}{L}\displaystyle\frac{d}{dc}\|\varphi_{c}\|^2 \mathcal{L}^{-1}1.
\end{equation}
Hence, since $\varphi_{c}$ has the mean zero property we have
\begin{eqnarray}\label{ilw.20}\displaystyle\left\langle\displaystyle\frac{\partial}{\partial c}\varphi_{c},1\right\rangle=\displaystyle\frac{\partial}{\partial c}\displaystyle\int_0^L\varphi_{c}(x)\ dx=0,
\end{eqnarray}
and so, by combining (\ref{ilw.19}) and (\ref{ilw.20}) it follows that
\begin{eqnarray}\label{ilw.21}\langle\mathcal{L}^{-1}\varphi_{c},1\rangle+\displaystyle\frac{1}{L}\displaystyle\frac{d}{dc}\|\varphi_{c}\|^2\langle\mathcal{L}^{-1}1,1\rangle=0.
\end{eqnarray}
Therefore, from (\ref{ilw.16}) and (\ref{ilw.21}) we arrive to the equality
\begin{eqnarray*}\langle\mathcal{L}^{-1}1,1\rangle+\displaystyle\frac{2}{Lc}\displaystyle\frac{d}{dc}\|\varphi_{c}\|^2 \langle\mathcal{L}^{-1}1,1\rangle=\displaystyle\frac{L}{c}.
\end{eqnarray*}
Lastly, since $\frac{d}{dc}\|\varphi_{c}\|^2>0$ (Proposition \ref{form1}), we get
\begin{eqnarray}\label{ilw.22}
\mathcal{I}=\langle\mathcal{L}^{-1}1,1\rangle=
\displaystyle\frac{L^2}{cL+2\displaystyle\frac{d}{dc}\|\varphi_{c}\|^2}.
\end{eqnarray}
Thus, we obtain the formula in (\ref{nI}) and from the hypotheses on $c$ and  Proposition \ref{form1} we have immediately that $\mathcal{I}>0$. This finishes the proof.
\end{proof}

\begin{obs} \label{numeric}  From (\ref{ilw.22}) we note that the requirement for $c$ to be positive in Proposition \ref{form2} has only technical reasons. If we do not require $c> 0$, the study of $\mathcal{I}$ will  depends on a ``heavy" numerical calculations. Here, additional calculations in \textit{Maple 16} enable us to say that $\mathcal{I}> 0$, for all $c \neq 0$  $(k\in (0, k_1)-\{k_0\})$.
  \end{obs}

\begin{prop}\label{form3}  For $c>0$ we obtain $\det(\mathcal D)<0$. In particular, $n(\mathcal D)=1$.
\end{prop}

\begin{proof}
We start by obtaining expressions for the elements of the matrix $\mathcal D$ in (\ref{Dmatrix}). Indeed,  from (\ref{ilw.16})  and (\ref{ilw.22}),
\begin{eqnarray}\label{ilw.23}\langle\mathcal{L}^{-1}[\varphi_{c}],1\rangle=\displaystyle\frac{c\langle\mathcal{L}^{-1}[1],1\rangle}{2}-\displaystyle\frac{L}{2}
=-\displaystyle\frac{L\displaystyle\frac{d}{dc}\|\varphi_{c}\|^2}{cL+2\displaystyle\frac{d}{dc}\|\varphi_{c}\|^2}.
\end{eqnarray}
Hence, by using identities (\ref{ilw.15.1}) and (\ref{ilw.23}) and the fact that $\varphi_{c}\in H_0$ we obtain
\begin{eqnarray}\label{ilw.24}\langle\mathcal{L}^{-1}\varphi_{c},\varphi_{c}\rangle=\displaystyle\frac{c}{2}\langle\mathcal{L}^{-1}1,\varphi_{c}\rangle=
-\displaystyle\frac{cL\displaystyle\frac{d}{dc}\|\varphi_{c}\|^2}{2cL+4\displaystyle\frac{d}{dc}\|\varphi_{c}\|^2}.
\end{eqnarray}

Then, since $\mathcal{I}\neq 0$ (Proposition \ref{form2}) follows from
(\ref{ilw.22}), (\ref{ilw.23}) and (\ref{ilw.24}) that
\begin{equation}\begin{array}{llllll}
\det(\mathcal{D})&=&\displaystyle\frac{1}{\mathcal{I}}\left[\langle\mathcal{L}^{-1}\varphi_{c},\varphi_{c}\rangle
-\displaystyle\frac{\langle\mathcal{L}^{-1}\varphi_{c},1\rangle^2}{\langle\mathcal{L}^{-1}1,1\rangle}\right]\\\\
&=&\displaystyle-\frac{1}{\mathcal{I}}\left[\displaystyle\frac{cL
\displaystyle\frac{d}{d
c}\|\varphi_{c}\|^2}{2cL+4 \displaystyle\frac{d}{d
c}\|\varphi_{c}\|^2}-\displaystyle\frac{\displaystyle\left[
\displaystyle\frac{d}{d
c}\|\varphi_{c}\|^2\right]^2}{cL+2\displaystyle\frac{d}{d
c}\|\varphi_{c}\|^2}\right]=-\displaystyle\frac{1}{2}\frac{1}{\mathcal{I}}
\displaystyle\frac{d}{d
c}\|\varphi_{c}\|^2.
\end{array}
\label{detD1}
\end{equation}
Therefore, we obtain the formula in (\ref{detD}) and from Propositions
\ref{form1}-\ref{form2} we have $\det(\mathcal D)<0$. This finishes
the proof of the Proposition.
\end{proof}

\begin{obs}
From $(\ref{vlinear2})$ and the fact that $\mathcal{L}=\mathcal{L}_{\varsigma,\delta}$, we deduce that the positive and periodic wave $\phi_{\varsigma}$ is also linearly stable.
\end{obs}

\section{Orbital Stability for the ILW-equation}

In the last section we have proved that the Krein-Hamiltonian index $\mathcal{K}_{{\rm
Ham}}$ associated to the linear operator $\partial_x\mathcal L$ is zero, and thus the linear stability of  the periodic traveling wave $\varphi_c$ was obtained. The next outcome of the theory is to obtain informations about the orbital stability of these periodic profiles. From the theories established in \cite{grillakis1}, \cite{grillakis2},  \cite{DK} and \cite[Chapter 5.2.2]{KP}, we can deduce that $\varphi_c$ will be a local minimizer of a constrained energy,  and so the orbital stability of these periodic waves is expected to be obtained provided we present a convenient global well-posedness result for the model (\ref{ILW}).

Now, the study of orbital stability can be based on  an analysis of Lyapunov type (see \cite{A},  \cite{Be}-\cite{bona1}-\cite{grillakis2}-\cite{grillakis1}-\cite{johnson09}-\cite{Wo}) and  it will work very well when the integration constant $A_c$ in (\ref{travkdv}) is constant  or zero. In the case of the integration constant $A_c$ to be a function of the wave velocity $c$, as in our case, it does not seem to be  immediate to apply this strategy. Thus, our following purpose will be to apply the recent development in Andrade and Pastor \cite{andrade-pastor} to handle such situations and so to obtain the orbital stability of the profile $\varphi_c$ for every $c\neq 0$ (see Theorem \ref{t.ilw.4.1} and Remark \ref{numeric})

We start our study by presenting the formal definition of orbital stability.
\begin{defi}\label{defi} We
say that the periodic wave $\varphi_c$ in (\ref{solilw}) is {\em orbitally stable} with respect to $(\ref{ILW})$ in the space $\mathcal W$ in (\ref{W}), if for all $\varepsilon >0$,  there exists $\delta >0$ such that if $u_0\in H^s_{per}([0,L])\cap \mathcal{W}$, $s>3/2$, with
$||u_0 - \varphi_c||_{\mathcal{W}} < \delta$ and $u(t)$ is the solution of
$(\ref{ILW})$ with $u(0) = u_0$, then for all $t\in\mathbb{R}$ one has
\[
\inf_{s\in\mathbb{R}}||u(t) -
\varphi_c(\cdot+s)||_{\mathcal{W}} < \varepsilon.
\]
Otherwise, the periodic wave $\varphi_c$ is said to be orbitally unstable.\label{defi1}\end{defi}

From Definition \ref{defi} we have that some information about the global well-posedness problem for the ILW-equation need to be established. That is the focus of the following theorem.

\begin{teo}\label{cauchy} Consider $u_0\in H^s_{per}([0,L])$. If
$s>\displaystyle\frac{3}{2}$, then there is a unique $u\in
C(\mathbb{R};H^s_{per}([0,L]))$, such that $u$ solves the initial value problem
\begin{eqnarray}\label{ilw.3.0.1}\displaystyle\left\{\begin{array}{l}
                                       u_t+2uu_x-(\mathcal{M}_{\delta}u)_x=0,\ \ (x,t)\in \mathbb{R}\times \mathbb{R}. \\
                                       u(0)=u_0.
                                     \end{array}
\right.
\end{eqnarray}
In addition, for all $T>0$ the mapping data-solution
$$
u_0\in H^s_{per}([0,L]) \to u\in C([0,T];H^{s}_{per}([0,L])),
$$
it is continuous.
\label{teoWP}\end{teo}
\begin{proof} See Abdelouhab {\it et al.} in \cite{abdel}.
\end{proof}

The ILW equation has the following three basic conserved quantities,
\begin{equation}\label{conser1}E_{-1}(u)=\displaystyle\int_0^L u\ dx,\ \ \ E_0(u)=\displaystyle\frac{1}{2}\displaystyle\int_0^Lu^2\ dx\end{equation} and
\begin{equation}\label{conser2}E_1(u)=\displaystyle\frac{1}{2}\displaystyle\int_0^L(\mathcal{M}_{\delta}u)u\ dx-\displaystyle\frac{1}{3}\displaystyle\int_0^Lu^3\ dx.
\end{equation}
Indeed, from Theorem $\ref{teoWP}$ and density arguments we deduce that for all $t$,
$$
E_{-1}(u(t))=E_{-1}(u_0),\;\; E_{0}(u(t))=E_{0}(u_0),\;\; and\;\; E_{1}(u(t))=E_{1}(u_0).
$$
Moreover, the ILW equation admits the following  Hamiltonian structure
$$
u_t=-2uu_x+(\mathcal{M}_{\delta}u)_x=\partial_x(-u^2+\mathcal{M}_{\delta}u)=\partial_xE_1'(u).
$$

Our purpose in the following is to describe Andrade and Pastor's approach \cite{andrade-pastor} in the case of the ILW equation. We note from Theorem \ref{exis} that the wave-velocity, $c$, of our periodic waves  in $(\ref{solilw})$ may also depend smoothly on the elliptic modulus $k$, ($k\to c(k)$, by equation (\ref{ilw.26})). Our stability analysis will be based on this parameter instead of the wave velocity parameter $c$, such as is standard in the classical literature. Therefore, we need to establish a stability framework based on this new ``wave-velocity'' parameter $k$. Thus, by following \cite{andrade-pastor} and \cite{grillakis1},  we consider for every $k\in (0, k_1)$ the following manifold in the space $\mathcal{W}$,
\begin{equation}\label{var1}
\Sigma_k=\left\{u\in\mathcal{W};\ M_k(u)=M_k(\varphi_k),\ \mbox{where}\ M_k(u):=\frac{dc}{dk}E_0(u)+\frac{dA}{dk}E_{-1}(u)\right\},
\end{equation}
where $\varphi_k=\varphi_{c(k)}$ and $A(k)=\frac{1}{L}\int_0^L\varphi_k^2(x)dx$. We note that the strategy established in \cite{andrade-pastor} is a generalization of the results in \cite{johnson09}.  The assumptions to obtain the orbital stability of $\varphi_k$ in the sense of Definition $\ref{defi1}$ and by depending of the parameter $k$  are the following:
\begin{enumerate}
\item[$(P0)$]\;\;There is a smooth curve of periodic solutions for (\ref{travkdv}) in the form,
$$
 k\in J\subset \mathbb R\to \varphi_k\in H^n_{per}([0,L])\cap H_0,\qquad n\in \mathbb N;
 $$
\item[$(P1)$] $ ker(\mathcal L)=[\frac{d}{dx} \varphi_k]$;
\item[$(P2)$] $ \mathcal L$ has an unique negative
eigenvalue $\lambda$, it which is simple;
\item[$(P3)$] $ \displaystyle\left\langle \mathcal{L}\left(\frac{\partial \varphi_k}{\partial k}\right),\left(\frac{\partial \varphi_k}{\partial k}\right)\right\rangle <0$.
\end{enumerate}

Conditions $(P0)-(P1)-(P2)$  have been established for us  in the Theorems \ref{exis} and \ref{ilw.spec} above. With regard to the condition $(P_3)$, if we derivate the equation in  (\ref{travkdv}) with regard to $k$ is obtained the relation
$$
\mathcal{L}\left(\frac{\partial \varphi_k}{\partial k}\right)=-\frac{dc}{dk} \varphi_k- \frac{dA}{dk}=-M_k'(\varphi_k).
$$
Thus, by Proposition \ref{form1}, Remark 4.1 and $\varphi_k\in H_0$ we obtain for  every $k$ such that $c=c(k)\neq 0$,
\begin{equation}\label{propr10}
\displaystyle\left\langle \mathcal{L}\left(\frac{\partial \varphi_k}{\partial k}\right),\left(\frac{\partial \varphi_k}{\partial k}\right)\right\rangle=-\left\langle M_k'(\varphi_k),\frac{\partial \varphi_k}{\partial k}\right\rangle=
-\frac{1}{2}\frac{dc}{dk}\frac{d}{dk}\int_0^L\varphi_k^2(x)dx<0.
\end{equation}

The main Theorem of this section is the following.

\begin{teo}\label{nonlinear} Let $k\in (0,k_1)$ be fixed such that $c=c(k)\neq 0$. Then the periodic wave  $\varphi_k=\varphi_{c(k)}$ in (\ref{solilw})   is orbitally stable by the periodic flow of the equation $(\ref{ILW})$ in the sense of Definition $\ref{defi1}$.
\end{teo}

By convenience of the reader we give a sketch  of the proof of Theorem \ref{nonlinear}. The proof of the following two Lemmas follow from the ideas in  \cite{andrade-pastor}, \cite{A}, \cite{grillakis1}, and \cite{johnson09}.

\begin{lema}\label{l.orb.ilw.2} There is $\varepsilon>0$ and a $C^1-$function, $\omega:U_{\varepsilon}(\varphi_k)\mapsto\mathbb{R}$, with
$$
U_{\varepsilon}(\varphi_k):=\{u\in \mathcal{W};\
\|u-\varphi_k\|_{\mathcal{W}}<\varepsilon\},
$$
such that
\begin{eqnarray*}\langle u(\cdot+\omega(u)),\frac{d}{dx}\varphi_k\rangle=0,\;\;\text{for all}\;\;\ u\in U_{\varepsilon}(\varphi_k).\end{eqnarray*}
\end{lema}

\begin{lema}\label{l.orb.ilw.3}  We consider the conditions $(P_0)-(P_1)-(P_2)-(P_3)$ above, and  the set
 $$
 \mathcal{A}_k:=\{\Phi\in \mathcal{W};\
\langle\Phi,M_k'(\varphi_k)\rangle=\langle\Phi,\frac{d}{dx}\varphi_k\rangle=0\}.
$$
Then, there exists a constant $C>0$ such that
\begin{eqnarray*}\langle \mathcal{L} \Phi,\Phi\rangle
\geq C\|\Phi\|^2_{\mathcal{W}},\quad\ \mbox{for all}\;\;
\Phi\in\mathcal{A}_k.\end{eqnarray*}
\end{lema}

Now, for $u\in \mathcal{W}$ we define the pseudo-metric
$$
\rho(u,\varphi_k):=\displaystyle\inf_{r\in [0,L]}\|u-\varphi_k(\cdot+r)\|_{\mathcal{W}},
$$
it which indicates the distance between $u$ and the orbit generated by $\varphi_k$ via the translation symmetry, namely, $\Omega_k=\{\varphi_k(\cdot+r): r\in [0,L]\}$.

The following Lemma establishes the local minimal property of the profile $\varphi_k$ on the manifold $\Sigma_k$.

\begin{lema}\label{ilw.40} We consider the conditions $(P_0)-(P_1)-(P_2)-(P_3)$ above, and we define the functional
$\mathcal{F}_k=E_1+cE_0+AE_{-1}$.
Then, there exist $\varepsilon>0$ and a constant $C(\varepsilon)>0$
satisfying
\begin{eqnarray*}\mathcal{F}_k(u)-\mathcal{F}_k(\varphi_k)\geq
C(\varepsilon)\cdot[\rho(u,\varphi_k)]^2,\end{eqnarray*} for all
$u\in U_{\varepsilon}(\varphi_k)\cap \Sigma_k$.
\end{lema}
\begin{proof} Consider $u\in \mathcal{W}$. Since $\mathcal{F}_k$ is invariant under translations one has
$\mathcal{F}_k(u)=\mathcal{F}_k(u(\cdot+r))$, for all
$r\in\mathbb{R}$. Thus, it is sufficient to show that
$$\mathcal{F}_k(u(\cdot+\omega(u)))-\mathcal{F}_k(\varphi_k)\geq
C\cdot[\rho(u,\varphi_k)]^2,$$ where $\omega$ is the smooth function obtained in Lemma \ref{l.orb.ilw.2}. Indeed, for $u\in  \Sigma_k$ follows from Lemma
\ref{l.orb.ilw.2} that there is a constant $C_1\in\mathbb{R}$ such that
\begin{equation}v:=u(\cdot+\omega(u))-\varphi_k=C_1M_k'(\varphi_k)+y,\label{v}\end{equation} where
$y\in\mathcal{B}_k=[M_k'(\varphi_k)]^{\perp}\cap[\frac{d}{dx}\varphi_k]^{\perp}$. Next, since $M_k$ is also invariant under translations we can apply Taylor's formula to obtain
\begin{eqnarray}\label{orb.ilw.12}
M_k(u)=M_k(u(\cdot+\omega(u)))=M_k(\varphi_k)+\langle
M_k'(\varphi_k),v\rangle+\mathcal{O}(\|v\|^2_{\mathcal W}).
\end{eqnarray}
Hence, since $y\in\mathcal{B}_k$ one has $\langle M_k'(\varphi_k),v\rangle=\langle M_k'(\varphi_k),C_1 M_k'(\varphi_k)\rangle=C_1 N$, where $N$ is a constant which is associated with the wave speed $c$. Then, since $M_k(u)=M_k(\varphi_k)$ we obtain immediately from (\ref{orb.ilw.12}) that
\begin{eqnarray}\label{C1}C_1=\mathcal{O}(\|v\|^2_{\mathcal W}).
\end{eqnarray}

Now, by applying a Taylor's expansion to $\mathcal{F}_k$ around
$u(\cdot+\omega(u))=\varphi_k+v$ we obtain
\begin{eqnarray*}\mathcal{F}_k(u)-\mathcal{F}_k(\varphi_k)=\displaystyle\frac{1}{2}\langle \mathcal{L}v,v\rangle+ o(\|v\|^2_{\mathcal W}),
\end{eqnarray*}
because of $\mathcal{F}_k'(\varphi_k)=0$ and
$\mathcal{F}_k''(\varphi_k)=\mathcal{L}$. By using $(\ref{v})$ and $(\ref{C1})$ we have
$\langle
\mathcal{L}v,v\rangle=
\langle\mathcal{L}y,y\rangle+\mathcal{O}(\|v\|^2_{\mathcal W})$, and so
we conclude $\mathcal{F}_k(u)-\mathcal{F}_k(\varphi_k)=\displaystyle\frac{1}{2}
\langle
\mathcal{L}y,y\rangle+ o(\|v\|^2_{\mathcal W})$. Next, since $y\in\mathcal{B}_k$, by Lemma \ref{l.orb.ilw.3} there is
$C>0$ such that $\langle\mathcal{L}y,y\rangle \geq
C\|y\|^2_{\mathcal{W}}$. Thus,
\begin{eqnarray}\label{orb.ilw.17}\mathcal{F}_k(u)-\mathcal{F}_k(\varphi_k)\geq
\widetilde{C}\|y\|^2_{\mathcal{W}}+o(\|v\|^2_{\mathcal W}),
\end{eqnarray}
where $\tilde{C}>0$. Therefore, from $(\ref{v})$ we deduce that for $\varepsilon>0$, small enough, there is $C=C(\varepsilon)>0$ such that
\begin{eqnarray*}\mathcal{F}_k(u)-\mathcal{F}_k(\varphi_k)\geq {C}\|v\|^2_{\mathcal{W}}\geq
{C}[\rho(u,\varphi_k)]^2.
\end{eqnarray*}
This finishes the proof.
\end{proof}

 \textbf{Proof of Theorem \ref{nonlinear}.} The proof of the result follows from Theorem \ref{cauchy}, Lemma \ref{ilw.40} and a convenient adaptation of  Theorem 3.5 in \cite{grillakis1} (see also \cite{andrade-pastor}). By contradiction, we can select $w_n:=u_n(\cdot,0)\in
U_{\frac{1}{n}}(\varphi_k)\cap H^{s}_{per}([0,L])$, $s>\frac{3}{2},$ and
$\varepsilon>0$, such that
$\|w_n-\varphi_k\|_{H^{s}_{per}}\stackrel{n\rightarrow\infty}{\longrightarrow}0$,
with $$\displaystyle\sup_{t\geq
0}\rho(u_n(\cdot,t),\varphi_k)\geq\varepsilon,$$ where $u_n(\cdot,t)$
is the corresponding solution of (\ref{cauchy}).
Let us consider $\varepsilon>0$ satisfying Lemma
\ref{l.orb.ilw.2}. From continuity of $u_n(t)$ at
$t\in \mathbb{R}$, we consider the smallest $t_n>0$ satisfying
\begin{eqnarray}\label{orb.ilw.20}\rho(u_n(\cdot,t_n),\varphi_k)=\displaystyle\frac{\varepsilon}{2}.
\end{eqnarray}

The following step in the analysis will be to determine the existence of $\alpha_n>0$ such that $\alpha_n u_n(\cdot, t_n)\in \Sigma_k$, for $n$ large. This is exactly the point in the theory that we will apply the strategy in \cite{andrade-pastor}. Indeed, let us define
$f_n:\mathbb{R}\rightarrow\mathbb{R}$, such that for $n$ fixed,
$$f_n(\alpha)=M_k(\alpha u_n(\cdot,t_n))=\displaystyle\frac{\alpha^2}{2}\displaystyle\frac{dc}{dk}\cdot\displaystyle\int_0^L|u_n(\cdot,t_n)|^2\ dx+\alpha\cdot\displaystyle\frac{dA}{dk}\cdot\displaystyle\int_0^Lu_n(\cdot,t_n)\ dx=:\alpha^2g_n+\alpha
h_n.$$
We note immediately that $f_n(0)=0$, $g_n>0$ and $M_k(\varphi_k)>0$. Thus, for all $n\in\mathbb{N}$ there exists
$\alpha_n>0$ such that $f_n(\alpha_n)=M_k(\varphi_k)$. In other words, there is
 $(\alpha_n)_{n\in\mathbb{N}}\subset\mathbb{R}$, satisfying
\begin{eqnarray}\label{orb.ilw.21}M_k(\alpha_nu_n(\cdot,t_n))=M_k(\varphi_k),\ \mbox{for all}\
n\in\mathbb{N},\end{eqnarray} that is,
$(\alpha_nu_n(\cdot,t_n))_{n\in\mathbb{N}}\subset \Sigma_k$.

Next, let  $\mathcal{T}_k(u):=\displaystyle\frac{dc}{dk} E_0(u)$
and $\mathcal{R}_k(u):=\displaystyle\frac{dA}{dk}
E_{-1}(u)$.  Then, since $E_0$ and $E_{-1}$ are continuous mapping
one has $\mathcal{T}_k(w_n)\longrightarrow\mathcal{T}_k(\varphi_k)=:g\neq 0$, $\mathcal{R}_k(w_n)\longrightarrow\mathcal{R}_k(\varphi_k)=:h$ and $M_k(w_n)\longrightarrow M_k(\varphi_k)$, as $n\rightarrow+\infty$.
So,
\begin{eqnarray*}\varrho_n&:=&|\alpha_n^2\mathcal{T}_k(w_n)+\alpha_n\mathcal{R}_k(w_n)-(\mathcal{T}_k(w_n)+\mathcal{R}_k(w_n))|\nonumber\\
&=&|M_k(\alpha_nu_n(\cdot,t_n))-M_k(w_n)|=|M_k(w_n)-M_k(\varphi_k)|\longrightarrow0.\nonumber
\end{eqnarray*}
as $n\rightarrow+\infty$.  On the other hand,
$$
0\leq|\alpha_n^2\mathcal{T}_k(w_n)+\alpha_n\mathcal{R}_k(w_n)-(g+h)|
\leq\varrho_n+|\mathcal{T}_k(w_n)-g|+|\mathcal{R}_k(w_n)-h|\longrightarrow 0,
$$
 that is, \begin{eqnarray}\label{orb.ilw.18}z_n:=\alpha_n^2\mathcal{T}_k(w_n)+\alpha_n\mathcal{R}_k(w_n)\longrightarrow g+h.
 \end{eqnarray}
Therefore, statement $(\ref{orb.ilw.18})$ gives us that $(\alpha_n)_{n\in\mathbb{N}}$ is a bounded sequence and therefore, modulo a subsequence, one has  $\alpha_n\longrightarrow\alpha_0$, as $ n\rightarrow+\infty$. We will see that $\alpha_0=1$.
Indeed, from (\ref{orb.ilw.18}) we get
\begin{eqnarray}\label{cond1}(1-\alpha_0)\cdot [(1+\alpha_0)\cdot
g+h]=0.
\end{eqnarray}
Now, since
$$
1+\displaystyle\frac{h}{g}=1+\displaystyle\frac{\mathcal{R}_k(\varphi_k)}{\mathcal{T}_k(\varphi_k)}=1+\displaystyle\frac{\displaystyle\frac{d
A}{d k}\displaystyle\int_0^L\varphi_k(\xi)\
d\xi}{\mathcal{T}_k(\varphi_k)}=1+\displaystyle\frac{0}{\mathcal{T}_k(\varphi_k)}=1>0,
$$
we obtain that $\alpha_0>0$. Therefore, since $g\neq 0$ follows from  $(\ref{cond1})$ that $\alpha_0=1$.\\
\indent Next, we claim that
\begin{equation}\label{af1}\rho(u_n(\cdot,t_n),\alpha_nu_n(\cdot,t_n))\longrightarrow0,\ \ \   n\rightarrow+\infty.\end{equation}
In fact, since $\rho(u_n(\cdot,t_n),\varphi_k)=\displaystyle\frac{\varepsilon}{2}$
there are $r_n\in\mathbb{R}$ and $C_2>0$ such that
\begin{eqnarray*}\|u_n(\cdot,t_n)\|_{\mathcal{W}}\leq
\|u_n(\cdot,t_n)-\varphi_k(\cdot+r_n)\|_{\mathcal{W}}+\|\varphi_k(\cdot+r_n)\|_{\mathcal{W}}<\varepsilon+\|\varphi_k(\cdot+r_n)\|_{\mathcal{W}}=C_2,\end{eqnarray*}
that is,
$\displaystyle\left(\|u_n(\cdot,t_n)\|_{\mathcal{W}}\right)_{n\in\mathbb{N}}$
is a bounded sequence. Therefore, the convergence $\alpha_n\to 1$ and the relation
\begin{equation}\label{orb.ilw.19}
\rho(u_n(\cdot,t_n),\alpha_nu_n(\cdot,t_n))\leq\|u_n(\cdot,t_n)-\alpha_nu_n(\cdot,t_n)\|_{\mathcal{W}}\leq
|1-\alpha_n|\cdot \|u_n(\cdot,t_n)\|_{\mathcal{W}},
\nonumber
\end{equation}
implies (\ref{af1}). Therefore, an application of the triangle inequality and $(\ref{orb.ilw.20})$
show that $(\alpha_nu_n(\cdot,t_n))_{n\in\mathbb{N}}\subset U_{\varepsilon}(\varphi_k)$. Hence, from Lemma \ref{ilw.40} we conclude immediately the convergence
 \begin{equation}\label{af2}\rho(\alpha_nu_n(\cdot,t_n),\varphi_k)\longrightarrow0,\ \ \ n\rightarrow+\infty.
 \end{equation}
Lastly,  by using $(\ref{af1})$ and $(\ref{af2})$ we obtain,
\begin{eqnarray*}\displaystyle\frac{\varepsilon}{2}=\rho(u_n(\cdot,t_n),\varphi_k)\leq
\rho(u_n(\cdot,t_n),\alpha_n u_n(\cdot,t_n))+\rho(\alpha_n u_n(\cdot,t_n),\varphi_k)\longrightarrow0,\ \ \ \ \ n\rightarrow+\infty,\end{eqnarray*}
which gives us a contradiction. The proof of Theorem $\ref{nonlinear}$ is now completed.
\begin{flushright}
$\square$
\end{flushright}

\begin{obs}
The positive and periodic wave $\phi_{\varsigma}$ in $(\ref{ilw.54})$ is orbitally stable by a direct application of the arguments in \cite{natali}.
\end{obs}

\textbf{Acknowledgements:} The authors are grateful to two anonymous referees for their valuable suggestions and constructive comments which greatly improved both the presentation and the scope of the paper. The research of J. Angulo and F. Natali was partially supported  by Grant CNPq/Brazil. E. Cardoso Jr. was supported by CAPES/Brazil.


\begin{thebibliography}{99}

\bibitem{abdel} Abdelouhab, L., Bona, J. L., Felland, M. and Saut, J-C., \textit{Nonlocal Models for Nonlinear Dispersive Waves.} Physica D,
40 (1989), pp. 360--392.

\bibitem{afss} Ablowitz, M.J., Fokas, A.S., Satsuma, J. and Segur, H., \textit{On the periodic intermediate long wave equation},  J. Phys. A: Math. Gen. 15 (1982), pp. 781--786.

\bibitem{abramo} Abramowitz, M. and Segun, I. A. Handbook of Mathematical Functions with Formulas, Graphs and Mathematical
Tables, Dover Publications, New York, 1972.
\bibitem{albert1} Albert, J.P., \textit{Positivity properties and stability of solitary-wave
solutions of model equations for long waves}, Comm. PDE, 17 (1992),
pp. 1--22.

\bibitem{AB}   Albert, J.P.  and Bona, J.L., \textit{Total positivity and the stability of
internal waves in fluids of finite depth}, IMA J. Applied Math. 46
(1991), pp. 1--19.

\bibitem{ABH} Albert, J.P., J.L. Bona, J.L., and Henry, D., \textit{Sufficient conditions for
stability of solitary-wave equation of model equations for long
waves}, Physica D 24 (1987), pp. 343--366.

\bibitem{andrade-pastor} Andrade, T. P. and Pastor, A., \textit{Orbital stability of periodic traveling-wave solutions for the BBM equation with fractional nonlinear term}, preprint (2015).

\bibitem{A}  Angulo, J.  \textit{Nonlinear Dispersive Equations: Existence and Stability of Solitary and Periodic Travelling Wave Solutions}, Mathematical Surveys and Monographs (SURV), 156,  AMS, (2009).

\bibitem{ABS} Angulo, J., Bona, J.L. and Scialom, M., \textit{Stability of cnoidal waves},
Advances in Differential Equations 11  (2006), p. 1321--1374.


\bibitem{AN1} Angulo, J. and Natali, F., \textit{Instability  of periodic traveling waves for dispersive models}, preprint (2012).
\bibitem{an} Angulo, J., Banquet, C., Silva, J.D. and Oliveira, F., \textit{The regularized boussinesq equation: instability of periodic
traveling waves}, J. Diff. Equat., 254 (2013) p. 3994-4023.
\bibitem{AN2}  Angulo, J. and Natali, F., \textit{Stability and instability of periodic
travelling waves solutions for the critical Korteweg-de Vries and
non-linear Schr\"odinger equations}, Physica D, 238 (2009), pp.
603--621.
\bibitem{natali} Angulo, J. and Natali, F., \textit{Positivity properties of the Fourier transform and the stability of periodic travelling-wave solutions}, SIAM J. Math. Anal., 40 (2008),
p. 1123--1151.
\bibitem{angulo3}
Angulo, J., Bona, J.L. and Scialom, M.  \textit{Stability of cnoidal
waves}, Adv. Diff. Equat., 11 (2006),  pp. 1321--1374.

\bibitem{Be} Benjamin, T.B., \textit{The stability of solitary waves}, Proc. Royal Soc. London Ser. A, 338
(1972), pp. 153--183.

\bibitem{benjamin1} Benjamin, T.B., \textit{Lectures on nonlinear wave motion}, Nonlinear Wave Motion, American Math. Soc., Lecture Notes in Applied
Mathematics 15 (1974), pp. 3--47.


\bibitem{bona1} Bona, J.L., \textit{On the stability theory of solitary waves}, Proc Roy. Soc.
Lond. Ser. A 344 (1975), pp. 363--374.



\bibitem{BJ} Bronski, J.C., and Johnson, M., \textit{The modulational instability for a generalized korteweg-de
vries equation}, Arch. Rat. Mech. and Anal.,  197 (2010), pp. 357--400.

\bibitem{BJK} Bronski, J.C., Johnson, M. and Kapitula, T., \textit{An index theorem for the stability of periodic traveling waves of KdV
type}, Proc. Roy. Soc. Edinburgh. Section A, 141 (2011), pp. 1141--1173.

\bibitem{BJK2} Bronski, J.C., Johnson, M. and Kapitula, T., \textit{An instability index theory for quadratic pencils and applications}, Comm. Math. Phys. 327(2) (2014), pp. 521--550.

\bibitem{BJZ}  Bronski, J. C., Johnson, M. and Zumbrun, K., \textit{On the modulation equations and stability
of periodic gkdv waves via bloch decompositions}, Physica D, 239 (2010), pp. 2057--2065.


\bibitem{byrd} Byrd, P.F. and Friedman, M.D., \textit{Handbook of elliptic integrals
for engineers and scientists}, 2nd ed., Springer, NY, (1971).

\bibitem{BSS} Bona, J. L., Souganidis, P. E. and Strauss, W. A., \textit{Stability and instability of solitary waves of Korteweg-de Vries type equations}, Proc. R. Soc. A 411 (1987), pp. 395--412





\bibitem{DK} Deconinck, B. and Kapitula, T.
\textit{On the spectral and orbital stability of spatially periodic stationary solutions
of generalized Korteweg-de Vries equations}. In P. Guyenne et al., Hamiltonian Partial Diff. Eq. Appl.,
75 (2015),  pp. 285--322.


\bibitem{DK1} Deconinck, B. and Kapitula, T. \textit{The orbital stability of the cnoidal waves of the Korteweg-de Vries equation}, Phys. Letters A, 374 (2010), pp. 4018--4022.

\bibitem{DN}  Deconinck, B. and Nivala, M. \textit{The stability
analysis of the periodic traveling wave solutions of the mKdV  equation},
Stud. Appl. Math. 126 (2010), pp. 17--48.

\bibitem{grillakis2} Grillakis, M., Shatah, J., and Strauss, W., \textit{Stability theory
of solitary waves in the presence of symmetry II}, J. Funct. Anal., 94 (1990), pp. 308--348.

\bibitem{grillakis1} Grillakis, M., Shatah, J., and Strauss, W., \textit{Stability theory
of solitary waves in the presence of symmetry I}, J. Funct. Anal.,
74 (1987), pp. 160--197.
\bibitem{hakka1} Hakkaev, S., Stanislavova, M. and Stefanov, A., \textit{Linear stability analysis for periodic traveling waves of the Boussinesq equation and
KGZ system}, Proc. Roy. Soc. Edinburgh A,  114 (2014), pp. 455--489

\bibitem{haragus} Haragus, M. and Kapitula, T., \textit{On the spectra of periodic waves for infinite-dimensional
 Hamiltonian systems}, Phys. D, 237 (2008), pp. 2649--2671.


 \bibitem{HJ} Hur, V. and Johnson, M., \textit{Stability of periodic traveling waves for nonlinear dispersive equations}, SIAM J. Math. Anal., 47 (2015), pp. 3528--3554.

  \bibitem{J2} Johnson, M. \textit{Stability of small periodic waves in fractional KdV-type equations}, SIAM J. Math. Anal., 45 (2013), pp. 3168--3193.
 \bibitem{J1} Johnson, M. \textit{On the stability of periodic solutions of the generalized Benjamin-Bona-Mahony equation}, Physica
D, 239 (2010), pp. 1892--1908.


 \bibitem{johnson09} Johnson, M., \textit{Nonlinear stability
of periodic traveling wave solutions of the generalized Korteweg-de
Vries equation}, SIAM J. Math. Anal., 41 (2009), pp. 1921--1947.

\bibitem{KP} Kapitula, T.  and Promislow, K., Spectral and Dynamical Stability of Nonlinear Waves, Springer, (2013).

 \bibitem{kap} Kapitula, T. and Stefanov, A., \textit{Hamiltonian-Krein (instability) index theory for KdV-like
eigenvalue problems}, Stud. Appl. Math., 132 (2014), pp. 183--211.
\bibitem{Kar} Karlin, S., \textit{Total Positivity}, Stanford University Press, (1968).
\bibitem{kato1}  Kato, T., \textit{Perturbation theory for linear Operators},
Springer, Berlin, 2nd ed., (1976).
 \bibitem{lin} Lin, Z., \textit{Instability of nonlinear dispersive solitary waves},
J. Funct. Anal., 255 (2008), pp. 1091--1124.
\bibitem{lopes} Lopes, O. \textit{A linearized instability result for solitary waves}.
Discrete Contin. Dyn. Syst. 8 (2002), pp. 115--119.

\bibitem{NM} Nakamura, A. and Matsuno, Y., \textit{Exact one-and two-periodic wave solutions of fluids of finite depth}, J. Phys. Soc. Jpn., 48 (1980), pp. 653--657.
\bibitem{parker} Parker, A., \textit{Periodic solutions of the intermediate long-wave equation: a nonlinear superposition
principle}, J. Phys. A: Math. Gen., 25 (1992), pp. 2005--2032.

\bibitem{stefa} Stanislavova, M. and Stefanov, A., \textit{Linear stability analysis for traveling waves of second order in time PDE's},
Nonlinearity, 25 (2012), pp. 2625--2654.



\bibitem{W1} Weinstein, M.I., \textit{Liapunov stability of ground
states of nonlinear dispersive evolution equations}. Comm. Pure
Appl. Math., 39 (1986), pp. 51--68.


\bibitem{Wo} Weinstein, M. I., \textit{On the structure and formation of singularities in solutions to nonlinear dispersive
evolution equations} Commun. Partial Diff. Equat. 11 (1986) pp. 545--65

\bibitem{W} Weinstein, M. I., \textit{ Nonlinear Schr\"odinger equation and sharp interpolation estimates} Commun. Math.
Phys. 87 (1983), pp. 567--76






\end{thebibliography}
\end{document}